\documentclass[12pt]{article}
\usepackage{amssymb,amsmath, enumerate, xcolor, amsfonts, hyperref, amsthm}
\usepackage[margin=1.2in]{geometry}
\usepackage{enumitem, enumerate}

\newtheorem{theorem}{Theorem}
\newtheorem{lemma}[theorem]{Lemma}
\newtheorem{corollary}[theorem]{Corollary}
\newtheorem{definition}[theorem]{Definition}

\newtheorem{conjecture}[theorem]{Conjecture}
\newtheorem{fact}[theorem]{Fact}
\newtheorem{claim}[theorem]{Claim}
\newtheorem{remark}[theorem]{Remark}

 \newcommand{\eps}{\varepsilon}       \newcommand{\W}{\mathcal W}                \newcommand{\V}{\mathcal V}                  \newcommand{\Q}{\mathcal Q}
\newcommand{\X}{\mathcal X} 
      \renewcommand{\epsilon}{\varepsilon}  
      \newcommand{\TO}{\T_{C'}^{Out}}

\newcommand{\ESC}{Erd\H os-S\'os conjecture~}

 \newcommand{\T}{\mathcal T}
 \newcommand{\uT}{\bigcup\mathcal T}
 \newcommand{\Td}{\mathcal T_{D'}}
 \newcommand{\uTd}{\bigcup\mathcal T_{D'}} 
 \newcommand{\Tc}{\mathcal T_{C'}}
 \newcommand{\uTc}{\bigcup\mathcal T_{C'}}

\newcommand{\es}{\emptyset}

\title{The Erd\H os-S\'os conjecture  in dense graphs} 
\author{
Bruce Reed\footnote{Mathematical Institute, Academia Sinica, Taiwan.  (\texttt{bruce.al.reed@gmail.com}).
Supported by  NSTC Grant 112-2115-M-001-013-MY3} 
\quad
Maya Stein\footnote{Departamento de Ingenier\'ia Matem\'atica y Centro de Modelamiento
Matem\'atico (CNRS IRL2807), Universidad de Chile, Santiago, Chile. Supported by ANID Regular Grant 1260024 and by ANID grant CMM Basal FB210005.\\ Parts of this work were conceived while both authors were in residence at the Simons Laufer Mathematical Sciences Institute in Berkeley, California during the Spring semester 2025, supported by the National Science Foundation under Grant No.~DMS-1928930.}
}
\date{}

\begin{document}
\maketitle
\begin{abstract}
 The Erd\H os--S\'os conjecture states that   every $n$-vertex graph with more than $(k-2)n/2$ edges contains every $k$-vertex tree. We prove that for every $\gamma$ there is an $n_0$ 
 such that for all $n\ge n_0$ and $k \ge \gamma n$ the conjecture holds.
 As a corollary of our result, we obtain a solution  of  a  51-year-old problem of
 Erd\H os and Graham on the multicolor Ramsey numbers of trees.
\end{abstract}

\section{Introduction}
The following famous conjecture has been open for over sixty years~\cite{Erdos64}. 

\begin{conjecture}[Erd\H os--S\'os conjecture]\label{es}
For $k, n\in\mathbb N^+$,  
every $n$-vertex graph with more than $(k-2)n/2$ edges contains every $k$-vertex tree as a subgraph.
\end{conjecture}

The conjecture is best possible because of the  clique on $k-1$ vertices which has no $k$-vertex subgraph. More generally, any $(k-2)$-regular graph serves as an example for Conjecture~\ref{es} being tight, as it fails to contain the $k$-vertex star.  

Many partial results have appeared since Conjecture~\ref{es} was posed (and a full proof was announced, but without a manuscript), 
 we refer to~\cite{maya-survey}~for  references. 
In particular, the conjecture was solved for  large trees of linearly bounded maximum degree  and large dense  host graphs by Besomi, Pavez-Sign\'e and one of the authors~\cite{BPS3}, and building on this, for all large trees of constant maximum degree by Pokrovskiy~\cite{Pok24}. Dropping the restrictions on the maximum degree of the tree, the conjecture was solved for all large trees and host graphs $G$ with $|V(G)|\le (1+10^{-1})k$ by the authors~\cite{RS25}. Very recently, Davoodi,  Piguet, {\v R}ada and Sanhueza-Matamala
gave a proof of the approximate version of the conjecture  for large dense graphs~\cite{beyond}. 

 We prove  the \ESC without approximation for large dense graphs:
\begin{theorem}
     \label{maint'}
     For each $\gamma>0$ there is an $n_0$ 
 such that for all $n\ge n_0$ and $k \ge \gamma n$,  every $n$-vertex graph $G$  with  average degree exceeding $k-2$ contains every tree $T$ on $k$ vertices as a subgraph.
 \end{theorem}

 We will use the main result from the companion paper~\cite{RSExt}, where    a graph is called  {\it robust}
  if
 all its proper subgraphs have strictly lower average degree.

\begin{theorem}\label{thm:realmain}
 There is a constant $\mu\in (0,1)$ 
 such that  for all $k\in\mathbb N$  each $k$-vertex tree 
  is contained in 
     each robust graph  of average degree exceeding $k-2$  which  has a subgraph $H$ of  minimum degree $\delta(H)\ge (1-\mu)k$.
 \end{theorem}

An overview of our proof of Theorem~\ref{maint'} is given in the next Section.\footnote{
We remark that it may be possible to deduce Theorem~\ref{maint'} by combining Theorem~\ref{thm:realmain} and results from Section~\ref{sec:hoststruc1} with the main result of~\cite{beyond}, but our proof of Theorem~\ref{maint'}, which was found independently of~\cite{beyond}, uses different ideas and is much shorter.
}
 

We close the introduction with an application of Theorem \ref{maint'} to Ramsey theory. 
The $\ell$-Ramsey number $R_\ell(G)$ of a graph $G$ is the minimum $n$ for which  every $\ell$-colouring of the edges of $K_n$ yields a monochromatic copy of $G$.
Burr and Roberts~\cite{BurrRoberts} determined $\ell$-Ramsey numbers of stars. 
 In 1975, Erd\H os and Graham~\cite{eg75} showed that 
 $R_{\ell}(T) > \ell(k-2)+1$ 
 for every $T$ on $k$ vertices and  sufficiently large $\ell\equiv  1 \bmod k$. They asked \footnote{We assume they mean that the $O(1)$-term is a constant that may depend on $\ell$.} whether $R_\ell(T)<\ell(k-1)+O(1)$ for every $\ell$ and for every $k$-vertex tree $T$, and noted that this would be implied by Conjecture~\ref{es}.  To see how Erd\H os and Graham's problem connects with Conjecture~\ref{es}, observe that for $\ell\ge 2$, every $\ell$-colouring of the edges of the complete graph on 
 $n=\ell (k-2)+2$ vertices  has a colour which is present on more than $\frac{k-2}2n$ edges. 
For fixed $\ell$, we can apply
Theorem~\ref{maint'} with $\gamma=1/\ell$ to find a monochromatic copy of any $k$-vertex tree, if $k$ is sufficiently large. So Theorem~\ref{maint'}  has the following corollary, which answers Erd\H os and Graham's question in the affirmative.
 \begin{corollary} For each $\ell\ge 2$ there is a $k_0$ such that $R_\ell(T)< \ell(k-2)+3$ for all $k\ge k_0$ and each $k$-vertex tree $T$.
 \end{corollary}

\subsection{Recent AI solution of the Erd\H{o}s--S\'{o}s conjecture}
It was announced very recently that GPT-6 Astra proved the Erd\H{o}s--S\'{o}s conjecture in full. Our proof was found without any use of AI. The version uploaded as the first arXiv version of this paper was ready in this form in early August 2026 (and a very similar version was ready in June 2026) but was not uploaded until the AI proof was announced. 
Although our result is now eclipsed by the AI proof, we believe that   the methods of this and the companion paper will likely have an impact on future work on related conjectures on tree containment.

\section{Overview of the proof}

We prove the dense case of the Erd\H{o}s-S\'{o}s conjecture by contradiction, assuming a minimal and thus robust counterexample $G$. The initial setup relies~on standard machinery for this type of problems: We decompose the tree $T$ into a constant sized vertex set $S$ and small rooted components of $T-S$ and apply Szemer\'{e}di's Regularity Lemma, with the aim of embedding the small trees into edges of the reduced graph (see Sections~\ref{sec:partitionTree} and~\ref{sec:regu}). We then identify  
several specific substructures in  the reduced host graph,  giving us an advantageous starting point for the embedding of $T$. For finding these structures, robustness and density of the host graph are crucial. Finally,  we employ a range of fine-tuned dynamic embedding strategies that are adapted to the particular configurations of both the host graph and the tree. 

\vspace{0.3cm}
\noindent\textbf{Structure in the host graph (special clusters $A$, $B$, $f$-matchings and obstructions).} 
The high average density of the host graph trivially ensures the existence
of a  cluster of degree at least $\kappa$  in the reduced graph (where $\kappa$ corresponds to $k$ scaled to the size of the reduced graph). With the help of Theorem~\ref{thm:realmain}, we can even find  clusters of degree $(1+\Omega(1))\kappa$ (see Lemma~\ref{mayaslemma1}). However,  such   high-degree clusters may be useless for tree embeddings if their neighbourhoods  are  dead ends. Moreover, we need two clusters to host the intersections $S_C, S_D$ of $S$ with the two colour classes of $T$.  Lemma~\ref{2+alpha} provides two adjacent clusters
$A,B$ satisfying
$d(A)+d(B)\ge(2+\Omega(1))\kappa$,
say with $d(A)\ge d(B)$. Moreover,  since $k$ is linear in $n$, we can ensure that  each cluster of $N(A)\cup N(B)$ has degree  $(1+\Omega(1))\frac \kappa 2$. We will always use   $A$, but not always use $B$; this will depend on the shape of the tree and on further structure of the host, which we explore next.

We  
isolate two $f$-matchings covering much of $N(A)$, or much of $N(A)\cup N(B)$, respectively (for an introduction to $f$-matchings 
 see Section~\ref{sec:match}). These $f$-matchings are found in Lemma~\ref{fracmatchlem}, and they are fitted to either using $A$, or both~$A$ and~$B$ in the later embedding. In the first case, we would ideally like to find an $f$-matching that uses about $d(W,A)f$ copies of $W$ for every cluster $W$. However, this may be impossible, and so we construct an $f$-matching~$M$ together with its stable set obstruction $Z$ (that is, $N(Z)$ is a bottleneck  to $M$ being as described above). Crucially,~$M$ matches all of $N(Z)$    to $Z$ in a way that it absorbs as much as possible of the degree of $A$ into~$Z$. Similarly,
we construct an $f$-matching $M_{AB}$ and its obstruction~$Z_{AB}$, with similar properties, but for the degree of $\{A,B\}$ into~$Z_{AB}$. 

Finally, in preparation for the  last and most delicate case of the embedding procedure,  we exploit robustness to extract a very special structure inside $Z$ and $N(Z)$, featuring a  cluster $B^*$ and its powerful neighbours.

\vspace{0.3cm}
\noindent\textbf{Our embedding strategies.} 
Most of the proof work is done during the embedding phases, where the forest of subtrees is packed into the regular pairs  corresponding to the edges of one of our  $f$-matchings. Our strategy diverges based on the nature of the structural obstructions.

\textbf{Case 1 ($Z = \emptyset$):} When the primary obstruction $Z$ is negligible or nonexistent, the matching $M$ almost perfectly covers the neighborhood of $A$. If also the secondary obstruction $Z_{AB}$ is non-existent, the embedding becomes an easy job if furthermore,~$A$ has very large degree (Case 1.1). Otherwise, $B$'s degree is large (Case 1.2). Then we have more freedom when allocating $S$ and our main worry is how to   balance the amount we embed into the two endpoints 
of an edge when packing small unbalanced trees  into it. For this we  need to embed with an eye 
on the ratio of the smaller versus the larger colour class (in each small tree), in a way that the imbalances cancel out. We achieve this by embedding into $B$ the one of $S_C, S_D$ whose adjacent trees have less imbalance overall, and embedding the other into $A$. 
We 
first embed  parts of the tree from $B$, then simultaneously from both $A$ and $B$, and   finish off as needed.

Next, if $Z_{AB}\neq\es$,  our job is easy if $A$ sees none of it (Case 1.3) as then we can embed from $B$ without disturbing the embedding from $A$ overly, but becomes more difficult if it does (Case 1.4). 
We crucially rely on $M_{AB}$ having been chosen earlier to intersect $M$ maximally, as this guarantees that $A$ has degree above~$\kappa$ into $V(M_{AB})$. 
Also, instead of embedding $S$ in $A\cup B$ we use $A$ and one of its neighbours $B'\in Z_{AB}$. The fact all neighbours of $B'$ are in $N(Z_{AB})$ allows us to fit $T$ into $M_{AB}$, much of it in edges between $N(B')$ and $Z_{AB}$. Again, we have to take care with skewed trees, with the difference that this time it turns out to be necessary to measure the imbalance by the size difference of the two colour classes instead of the ratio.   

\textbf{Case 2 ($Z \neq \emptyset$):} When $Z$ is large, its neighbourhood $N(Z)$ is a bottleneck (unless the size of $N(Z)$ is decent enough compared to the tree size, which provides an easy embedding situation and is treated at the beginning of the case). The remainder of Case 2 splits into two cases, depending on the shape of $T$.
In Case 2.1, we  overcome the bottleneck  by embedding part of the small trees in an `inverted way', i.e.~we embed roots that would normally go to $Z$ in $N(Z)$ instead. We  manage to fill~$Z$ and~$N(Z)$ sufficiently to eventually break out from the obstruction using $A$'s large degree.

In Case 2.2, we treat trees $T$ for which the above approach fails. Here we need the  special cluster $B^* \in Z \cap N(A)$  mentioned above.  A precise amount of the more unbalanced trees avoids $Z$ and is  instead routed through $B^*$'s powerful neighbourhood~$\Q$ into the rest of the graph, while still avoiding~$\Q$. When everything but $\Q$ is filled, 
the final embedding stage is delicate.
We  carefully distribute the little available space inside $\Q$ between the two partition classes of the remainder of $T$. Timing is important, too: we embed the root of each small tree when organising the space, but  coordinate the order of the remaining embedding meticulously.

\section{Preliminaries}\label{sec:prelim}
In this section, we will state some known and widely used tools for tree embeddings, and sharpen these tools by proving some auxiliary results. We group the preliminaries according to their theme into  subsections.

\subsection{Partitioning the tree}\label{sec:partitionTree}

The following definition is central to our proof.
\begin{definition}[$\alpha $-decomposition of a tree]\label{betadecomp}
Let $T$ be a tree  on $k\ge 1$ vertices with root~$r$. 
    For $\alpha>0$,  an $\alpha$-decomposition of  $T$  is a pair $(\mathcal T, S)$ with $S\subseteq V(T)$ and $\mathcal T$ being the set of components of $T-S$, rooted at their vertex closest to $r$, such that the following holds for   some labelling $(C,D)$ of the colour classes of $T$: 
\begin{enumerate}
\item[(a)] $r\in S$ and $|S| \le \frac{20}{\alpha^2}$,
\item[(b)] for each $X\in\mathcal T$, $|X|\le \alpha  k$  and $X$ is adjacent to at most two elements of $S$.
\end{enumerate}
Letting $\mathcal T_C$  consist of all $X\in\mathcal T$ whose root has its parent in $S_C:=S\cap C$, defining  $\mathcal T_D$ and $S_D$ analogously, and letting $\mathcal T_2$ consist of those $X\in\mathcal T$ that are adjacent to two vertices of $S$, 
we further require the following:
\begin{enumerate}
\item[(c)] $\mathcal T_2\subseteq \mathcal T_C$ and for  each $X\in\mathcal T_2$,
its two neighbours  in $S$ 
have even distance at least~six,
\item[(d)] $|\bigcup\mathcal T_C\setminus \bigcup \mathcal T_2|\ge |\bigcup\mathcal T_D|$.
\end{enumerate}
\end{definition}

\begin{lemma}{\bf\upshape\cite{LKS4}}
\label{treepartlem}
 For any $\alpha>0$  each tree $T$ with $k\ge 1$ vertices has an $\alpha$-decomposition.  
 \end{lemma}


\subsection{Regularity}\label{sec:regu}

 We present some standard machinery of the regularity method. 
 Let $G$ be a graph. The density  $dens_G(A,B)$ of
 a pair of disjoint $A,B\subseteq V(G)$ 
 is $\frac{e(A,B)}{|A||B|}$, where $e(A,B)$ is the number of edges from $A$ to $B$. 
 We say $A'\subseteq A$ is \textit{$\varepsilon$-significant} for $A$ (or  $\epsilon$-significant if $A$ is clear) if $|A'|> \varepsilon |A|$.
 For any $\eps>0$  the  pair $(A,B)$ is
 $\epsilon$-regular if for all  $\epsilon$-significant $A' \subseteq A$, $B' \subseteq B$ we have that $|dens_G(A',B')-dens_G(A,B)| < \epsilon$. 
A vertex $x\in A$ is  \textit{$\varepsilon  $-typical} to a set $Y\subseteq B$ if  $ (dens_G(A,B)-\varepsilon)|Y|\le \deg(x,Y) \le (dens_G(A,B)+\varepsilon)|Y|$. We simply write \textit{regular}, \textit{significant} or \textit{typical} if $\varepsilon$ is clear from the context.

In many ways regular pairs behave like  random bipartite graphs with the same edge density. 
The next well known fact (see for instance~\cite{regu}) captures this for the properties of  typicality and subpair regularity inheritance.

\begin{fact}\label{fact:1} 
For all $0< \eps\le \gamma\le 1$ and any  $\varepsilon$-regular pair $(A,B)$ we have:
\begin{enumerate}
 	\item[(i)] For all $\varepsilon$-significant $Y\subseteq B$,  at most  $2\varepsilon|A|$ vertices  of $A$ are not $\varepsilon$-typical~to~$Y$.
 	\item[(ii)]\label{fact:1,2}  For   all $\gamma$-significant  $X\subseteq A$, $Y\subseteq B$, the pair~$(X,Y)$ is $\frac{2\varepsilon}{\gamma}$-regular with density~$d$ obeying  $dens_G(A,B)-\varepsilon\le d \le dens_G(A,B)+\varepsilon$.
\end{enumerate}
\end{fact}

Call a vertex partition $V(G)=V_1\cup\dots\cup V_\ell$ an {\em $(\varepsilon,\rho)$-regular partition} if 
\begin{enumerate}
	\item[(i)] $|V_1|=|V_2|=\dots=|V_\ell|$;
	\item[(ii)]  $V_i$ is independent for all $i\in[\ell]$; and
	\item[(iii)] for all $1\le i<j\le \ell$, the pair $(V_i,V_j)$ is $\varepsilon$-regular, and either $dens_G(V_i,V_j)>\rho$ or $dens_G(V_i,V_j)=0$.
\end{enumerate}

The following simple fact is 
 well-known (see e.g.~\cite{LKS4}).
 \begin{fact}\label{fact:typical}
      For 
      each part $V_i$ of an $(\eps,\rho)$-regular partition, and   for any set $\mathcal X$ of $\eps$-significant subsets of clusters $V_j$ with $j\neq i$, all but
      at most $2\sqrt{\eps}|V_i|$ vertices of $V_i$ are 
      typical with respect to all but at most $\sqrt \eps|\mathcal X|$ of the sets in $\mathcal X$.
 \end{fact}

 Szemer\'edi's  regularity lemma~\cite{Sze78} states that for any  $\varepsilon>0$,  any sufficiently large  graph admits a regular partition of bounded size. 

\begin{lemma}[Szemer\'edi's regularity lemma - Degree form~\cite{regu}]\label{reg:deg}
For all $\varepsilon>0$  there are $n_0, M$ with $n_0,M\ge\frac 1\eps$ such that the following holds for all $\rho\in[0,1]$, all $n\ge n_0$, and any $n$-vertex graph $G$. There is a subgraph $G'$ of $G$ such that
\begin{enumerate}
\item  $|V(G)|- |V(G')|\le \varepsilon n$,  
\item  for all $x\in V(G')$ we have that $\deg_{G'}(x)\ge \deg_G(x)-(\rho+\varepsilon)n$, and
\item $G'$ admits an $(\varepsilon,\rho)$-regular partition into $\ell$ parts, with $\frac 1\eps\le \ell\le M$. 
\end{enumerate}
\end{lemma}

The {\em $(\varepsilon,\rho)$-reduced graph $R$ of $G$}  with respect to $G'$ and the $(\varepsilon,\rho)$-regular partition given by Lemma~\ref{reg:deg}   is the graph on vertex set $\mathcal V=\{V_i:i\in[\ell]\}$, where $V_iV_j$ is an edge if and only if $dens_{G'}(V_i,V_j)>0$. Although a reduced graph is associated to $G'$ and a partition of $G'$ (and different $G'$s or partitions could give rise to different reduced graphs), we often omit the reference to $G'$ and~the~partition.
We  assign weights to the edges of $R$: Namely, each edge~$V_iV_j$ has weight $d(V_i, V_j)=dens_{G'}(V_i, V_j)$.
 If $\mathcal W\subseteq \mathcal V$ then let $d(V_i,\mathcal W):=\sum_{W\in\mathcal W}d(V_i,W)$ and   $d(V_i):=d(V_i, \mathcal V)$.

A reduced graph~$R$ inherits many  properties of $G$, such as the edge density or the minimum degree (scaled to the order of $R$). Indeed, letting $\delta(R)$ and $d(R)$ denote the minimum and average {\it weighted} degree of $R$,   the following is well-known (see e.g.~\cite{regu}).

\begin{fact}\label{fact:2} 
Let $0<2\varepsilon\le\rho\le \min\{\tfrac{\alpha}{2}, \tfrac{\gamma}{2}\}$ and let $G$ be an $n$-vertex graph with minimum degree at least $ {\alpha} n$ and average degree at least $ \gamma n$. Then any $(\varepsilon,\rho)$-reduced graph $R$ of~$G$ satisfies  $\delta(R)\ge ( {\alpha}-2\rho)|R|$ and $d(R)\ge(\gamma-2\rho)|R|$. 
	\end{fact}

The following lemma encapsulates the standard approach for  embedding a small tree  into a regular pair. Similar statements have appeared e.g.~in~\cite{bps}. 

\begin{lemma}\label{lem:T1}
Let $0<\beta\le \varepsilon \le \tfrac{1}{250}$. Let $(V_1, V_2)$ be an $(\varepsilon,\varepsilon^{1/5})$-regular pair  with $|V_1|=|V_2|=m$, and  for $i=1,2$, let $V_i'\subseteq V_i$ be such that $|V_i'|>\varepsilon^{1/3}m$. Let $v\in V_1$ 
 see at least $\eps^{1/10}m$ vertices of $V_2'$. 
Then any tree $T$ on at most $\beta m$ vertices can be embedded into $V_1'\cup\{v\}\cup V_2'$, with its root $r$ mapped to $v$.  Furthermore, if $r'\in V(T)$ is at even distance at least four from $r$, then  $r'$ can be mapped to any 
 $v'\neq v$ of~$V_1'$ that sees at least $\eps^{1/10}m$ vertices of $V_2'$. 
\end{lemma}
\begin{proof}
First assume $r'$ does not exist. We
 construct the embedding   levelwise, starting by mapping $r$ to $v$. For the first level, we use that
$\deg(v,V_2')\ge\eps^{1/10}m$ to embed  the children of
$r$ in vertices typical to the unused part of $V_1'$.
All subsequent steps proceed levelwise as follows. At each step $i$ we ensure that all vertices of level $i$ are embedded into unused vertices of $V_j'$ that are typical   to the unoccupied vertices of $V_{3-j}'$, for the appropriate $j\in\{1,2\}$. This is possible as by Fact~\ref{fact:1}~(i),  at each step $i$, the degree of a typical vertex into the unoccupied vertices on the other side is at least $10\eps m$, while at most $2\eps m+\beta m$  vertices are non-typical or already occupied.

 Now assume $r'$ exists, and lies at even distance $p\ge 4$ from $r$. Similarly as above, we   construct a path of length $p-3$ that starts at $v$ and  alternates between $V_1'$ and $V'_2$, always embedding into vertices of $V_j'$ that are typical   to the unoccupied vertices of $V_{3-j}'$, for the appropriate $j\in\{1,2\}$. The last vertex $w$ of this path is in $V'_2$, and by Fact~\ref{fact:1}~(ii), the pair $(N(w)\cap V'_1, N(v')\cap V'_2)$ has density at least $\eps^{1/4}$ and thus spans at least one edge between vertices that are unused and typical to the respective $V'_j$. So we can complete the $v$--$v'$ path, and  embed the remainder of~$T$ in $V_1'\cup V'_2$ similarly as in the previous paragraph, but component by component.  
\end{proof}

\subsection{Subgraphs of robust graphs}\label{sec:hoststruc1}

We start with a well-known folklore fact, proven by repeatedly deleting vertices 
of degree less than $\frac k2$.
\begin{lemma}\label{folklore1}
    Each graph $H$ of average degree~$d(H)>k$ has a subgraph of minimum degree at least $\frac k2$ and average degree at least $d(H)$. 
\end{lemma}

Recall that in the introduction, we defined robust graphs as follows: If for every  $H\subsetneq G$, we have $d(H)<d(G)$ then $G$ is {\it robust}. Because of Lemma~\ref{folklore1}, 
 every  vertex  of a robust graph $G$ with
  $d(G)>k-2$ has degree at least $\frac{k-1}{2}$. We can generalize this observation by considering a set $S\subsetneq V(G)$ instead of a vertex. 
 \begin{lemma}\label{robustlemma}
   Let $G$ be a robust graph with $d(G)>k-2$ and let $\es\neq S\subsetneq V(G)$. Then $\sum_{v\in S}d_G(v)+ e(S, G-S)>(k-2)|S|$.  
 \end{lemma}
 \begin{proof}
     Indeed, otherwise $d(G-S)>d(G)$, a contradiction to $G$ being robust.
 \end{proof}

In a different direction, 
the next lemma ensures that we can do slightly better than Lemma~\ref{folklore1} if $k$ is linear in $n$. We state it for weighted graphs as we will apply it to the reduced graph of our host graph. 

\begin{lemma}\label{mindegsubgr}
Let $0<\alpha<\frac 14$ and $k,n\in \mathbb N$ be such that $k\ge \alpha n$.
    Let $H$ be a weighted $n$-vertex graph with weights in $[0,1]$. If $d(H)>(1-\alpha^3)k$ then $H$ has a subgraph $H'$ with $\delta(H')\ge (1+\alpha^3)\frac k2$ and $d(H')>(1-\alpha)k$. 
\end{lemma}
\begin{proof}
For any subgraph $H'$ of $H$ let $e^w_{H'}$ be the sum of the weights of the edges of~$H'$. In particular, $e^w_H>(1-\alpha^3)k\frac n2$ by assumption. Successively delete from $H$ any vertex of weighted degree below $(1+\alpha^3)\frac k2$ as long as this is  possible. The remaining graph~$H'$ has  $\delta(H')\ge (1+\alpha^3)\frac k2$  unless we deleted all vertices of~$H$. In the latter~case,    \begin{align*}(1-\alpha^3)k\frac n2< e^w_H &\le (n-\alpha k)(1+\alpha^3)\frac k2  +(\alpha k)^2 
    \le (1+\alpha^3)k\frac n2  -\frac\alpha 2 k^2+\alpha^2 k^2,
    \end{align*} 
    where the second inequality relies on the fact that at the moment of its deletion, every vertex $v$ has degree below $(1+\alpha^3)\frac k2$, and even below $\alpha k$ if $v$ is one of the last $\alpha k$ vertices to be deleted.  
    So as $k\ge \alpha n$, we obtain
     $\alpha^3 k  n 
     \ge \frac\alpha 2 k^2-\alpha^2 k^2\ge \frac\alpha 4 k^2 \ge \frac{\alpha^2}4 kn, $     
  which  contradicts the fact that $\alpha<\frac 14$.

     It only remains to see that $d(H')>(1-\alpha)k$. For this it suffices to observe that \begin{align*}
      e^w_{H'} - (1-\alpha)\frac{k}2|H'|
      & >(1-\alpha^3)k\frac n2- (n-|H'|)(1+\alpha^3)\frac k2 - (1-\alpha)\frac{k}2|H'|
\\ &   \ge 
-\alpha^3kn + \alpha|H'|\frac k2 
\\ &   \ge 
-\alpha^3kn + \delta(H')\frac{\alpha^2 n}2 
\ \ge \
(\frac{\alpha^2}4-\alpha^3)kn 
\ > \
     0.
     \end{align*}
     \vskip-.9cm
\end{proof}

Next, we show that any graph of average degree close to $k$ either has a substantial number of vertices of degree substantially above $k$ or contains every $k$-vertex tree $T$ as a subgraph. For the proof we need Theorem~\ref{thm:realmain}. 
\begin{lemma}
\label{mayaslemma1}
Let $0<\alpha<1$, $k,n\in \mathbb N$ be such that $k\ge \alpha^{1/10}  n$ and  
$\alpha^{1/50}\le \frac\mu{100}$ where~$\mu$ is the 
constant from Theorem~\ref{thm:realmain}. Let $H$ be 
    an  $n$-vertex subgraph   with $d(H) \ge (1-\alpha)k$   of a robust graph $G$ with  $d(G)>k-2$. Then  at least one of the following~holds:
    \begin{enumerate}
        \item[(a)] $G$ contains each $k$-vertex tree; or 
        \item[(b)] $H$ has  at least $\alpha^{1/4} n$ vertices $v$  with $d(v)>(1+\alpha^{1/4}) k$. 
    \end{enumerate}
\end{lemma}

\begin{proof}
     Let $X$ be the set of all
      vertices of $H$ having degree less than $(1-\alpha^{1/41} )k$, and let $Y$ be the set of all
      vertices of $H$ having degree above $(1+\alpha^{1/4})k$.  Assume (b) does not hold, i.e., $|Y|<\alpha^{1/4} n $.
     Then as  
      $\Delta(H)-d(H)\le \Delta(H)\le n\le \frac k{\alpha^{1/10}}$, we have
     \begin{align*}
         |X|\cdot \alpha^{1/40} k \le |X|\cdot (\alpha^{1/41}-\alpha) k &\le |Y|\cdot \big (\Delta(H)-d(H)\big )+ n(\alpha^{1/4}+\alpha) k \\ &< \alpha^{1/4} (\frac k{\alpha^{1/10}})^2+  \frac k{\alpha^{1/10}}\cdot 2\alpha^{1/4} k
        \  \le  \ 2 \alpha^{1/20} k^2.
     \end{align*}
   Thus $|X|\le 2
      \alpha^{1/40} k.$ Deleting $X$ from $H$ leaves a subgraph of  minimum degree exceeding $(1-\alpha^{1/41} -2\alpha^{1/40})k\ge (1-3\alpha^{1/41})k$. So by Theorem \ref{thm:realmain}, 
    (a) holds. 
\end{proof}

The next lemma shows we can find a subgraph of non-vanishing density between the set of high-degree vertices from the previous lemma and the rest.

\begin{lemma}
\label{twosides}
Let $\alpha>0$, $k,n\in \mathbb N$ be such that $k\ge \alpha^{1/10} n$ and  $\alpha^{1/50}\le \frac\mu{100}$  where $\mu$ is the 
constant from Theorem~\ref{thm:realmain}. Let $H$ be
    an  $n$-vertex subgraph with $d(H)>(1-\alpha)k$ of a robust graph $G$ with $k-2<d(G)<k$. Then at least one of the following holds:
    \begin{enumerate}
        \item[(a)] $G$ contains each $k$-vertex tree; or
         \item[(b)] $H$ contains a bipartite subgraph $(A_1,A_2)$ such that  
        \begin{enumerate}
        \item[(1)] $d_G(v)>(1+\alpha^{1/4}) k$ for each $v\in A_1$ and
        \item[(2)]  for $i=1,2$, each vertex in $A_i$ has at least $\alpha^{1/2}k$ neighbours in~$A_{3-i}$. 
          \end{enumerate}  
    \end{enumerate}
\end{lemma}

\begin{proof}
    Assume (a) does not hold and let $A'_1$ be a set of $\lceil \alpha^{1/4} n\rceil$ vertices~$v$  with $d(v)>(1+\alpha^{1/4}) k$, as given by  Lemma~\ref{mayaslemma1}~(b). Set $A'_2=V(H)\setminus A'_1$. Since $G$ is robust and $H[A_1']\subseteq G$, the average degree of $H[A'_1]$ is at most~$k$. So there are at least $|A'_1|\cdot \alpha^{1/4} k\ge \alpha^{1/2}nk$  edges between $A'_1$ and $A'_2$, and hence the average degree of  $(A'_1,A'_2)$ is at least $2\alpha^{1/2} k$. By Lemma~\ref{folklore1}, the bipartite graph between $A'_1$ and $A'_2$ has a subgraph $(A_1,A_2)$ of minimum degree at least $\alpha^{1/2}k$, which is as desired for~(b).
\end{proof}

The last lemma of this section provides two vertices whose degree  sum is large.

\begin{lemma}\label{2+alpha}
Let
$0<\eps\ll\rho\ll \alpha<1$, $k,n\in \mathbb N$ with $k\ge \alpha^{1/10}n$ and  $\alpha^{1/50}\le \frac\mu{100}$  where $\mu$ is the  
constant from Theorem~\ref{thm:realmain}.
 Let $R$ be an induced subgraph of an $(\eps, \rho)$-reduced graph on clusters of size $m$ of  a robust $n$-vertex graph $G$  such that $k-2< d(G)\le k$,
 and $d(R)\ge (1-\alpha)\kappa$ 
 where $\kappa=\frac{k}{m}$. Then
at least one of the following holds:
\begin{enumerate}
    \item[(a)]  There are adjacent  $A, B\in V(R)$ such that $d(A)+d(B)\ge (2+2\alpha)\kappa$, or
    \item[(b)] $G$ contains each  $k$-vertex tree.
\end{enumerate}
\end{lemma}

\begin{proof}
Assume (a)  fails. 
Set $e_R^w:=\sum_{XY\in E(R)}d(X,Y)$
 and use the Cauchy-Schwarz type equality  $ ( \sum_{i\in[\ell]} a_i   )^2 = \ell\sum_{i\in[\ell]}a_i^2   - \frac{1}{2} \sum_{i,j\in[\ell]}\big(a_i - a_j )^2$  to calculate
 \vskip-.2cm
\begin{align*}
(2+2\alpha)\kappa \cdot e_R^w & \ > \sum_{XY\in E(R)}d(X,Y)(d(X)+d(Y))
\\ & 
= \sum_{X\in V(R)}(d(X))^2 \\
& = \frac 1{|V(R)|}\Big( \big(\sum_{X\in V(R)}d(X)\big)^2   +  \frac 12\sum_{X,Y\in V(R)}\big(d(X)-d(Y)\big)^2 \Big)\\ 
& = \frac{|V(R)|d(R)\cdot 2e_R^w}{|V(R)|}   +  \frac 1{2|V(R)|}\sum_{X,Y\in V(R)}\big(d(X)-d(Y)\big)^2  \\ 
& \ge 2(1-\alpha)\kappa \cdot e_R^w  +  \frac 1{2|V(R)|}\sum_{X,Y\in V(R)}\big(d(X)-d(Y)\big)^2.
\end{align*}
 So, 
\begin{equation}\label{sqarediffs}
    \sum_{X,Y\in V(R)}\big(d(X)-d(Y)\big)^2 <8\alpha\kappa |V(R)|\cdot e_R^w\le 16\alpha^{9/10}\kappa^3|V(R)|,
\end{equation}
where in the last inequality we use that $e_R^w\le (1+\alpha)\kappa |V(R)|\le\frac{2\kappa^2}{\alpha^{1/10}}$
 which holds since the robustness of $G$ implies that $G':=G[\bigcup V(R)]\subseteq G$ has 
  average degree at most $(1+\alpha^2)k$. 
Now, 
if (b) fails, then  Lemma~\ref{mayaslemma1}  gives a set of $\alpha^{1/4} |V(G')|$ vertices of~$G'$ of degree at least $(1+\alpha ^{1/4})k$. Since moreover, a standard calculation gives that almost all vertices of a cluster have roughly the same degree, this means that there is a set $L$ of at least $\alpha^{11/40}|V(R)|$ clusters with $d_R(W)\ge (1+\alpha^{11/40})\kappa$ for each $W\in L$. As we assume $(a)$ fails,  each of the at least $(1+\alpha^{11/40})\kappa\ge\kappa$ clusters of $N(L)$ has degree below $(1-\alpha^{11/40}+2\alpha)\kappa$.~Hence
\begin{align*}
    \sum_{X\in L, Y\in N(L)}\big (d(X)-d(Y)\big)^2
     \ge (2\alpha^{11/40}-2\alpha)^2\kappa^2 \cdot \alpha^{11/40}|V(R)| \cdot \kappa 
\ge \alpha^{33/40}\kappa^3|V(R)|,
\end{align*}
a contradiction to~\eqref{sqarediffs}.
\end{proof}

 \subsection{$f$-matchings}\label{sec:match}

As usual, we call a matching  of a graph $H$   {\it perfect} if its edges cover all vertices of~$H$, and {\it near perfect} if its edges cover all but one vertex of $H$. If for each $v\in V(H)$, the graph $H-v$ has a perfect matching, then   $H$ is {\it factor-critical}.
We will use the following famous theorem:

\begin{theorem}[Gallai-Edmonds (see~\cite{LovaszPlummer})]\label{GEthm}
The vertex set of any graph $H$ can be partitioned into three sets $X$, $Y$ and $W$ such that 
\begin{enumerate}
\item[(I)] there are no edges between $Y$ and $W$,
    \item[(II)] each component of $H[W]$ has a perfect matching,
    \item[(III)] every component  of $H[Y]$  is factor-critical,
    \item[(IV)] every maximum  matching of 
$H$  consists of  near-perfect matchings of each component of $H[Y]$, a 
 perfect matching of $H[W]$, and a matching of all of $X$ to~$Y$.
\end{enumerate}
\end{theorem}

An {\it $f$-matching}  in a graph $H$ is a multiset $M$ of edges so that each vertex of $H$ is contained in at most $f$ edges of $M$. A vertex is {\it covered} by $M$ if it is contained in exactly $f$ edges. It is {\it touched} by $M$ if it is in at least one edge of $M$. If all vertices of $H$ are covered, then $M$ is called {\it perfect}. Note that every (usual) matching is a $1$-matching. We write $V(M)$ for the set of all vertices touched by~$M$.

An {\it $M$-alternating path} for an $f$-matching $M$ is a path $P$ such that 
if $k\le |E(P)|$ is even, then the $k$th 
edge on $P$ has  multiplicity at least $1$ in $M$. 
Consider  an $M$-alternating path $P=v_0v_1\ldots v_{2\ell-1}v_{2\ell}$  on an even number of edges,  starting at a vertex that is not covered by $M$. 
Obtain $M'$ from $M$ by  discarding from $M$ one copy of each edge  $v_{2i+1}v_{2i+2}$  and adding 
one copy of each edge  $v_{2i}v_{2i+1}$ to the multiset, for $i=0,\ldots \ell-1$. Clearly,  also $M'$ is an $f$-matching of size $|M|$.

\begin{lemma}[Tutte~\cite{Tut}]\label{factorcrit}
    Every factor-critical graph $G$ has a perfect $2$-matching.
\end{lemma}

Let $d_M(v)$ denote the number of edges of $M$ that contain $v$.
The next  lemma is crucial for our embedding strategy later.

\begin{lemma}\label{fracmatchlem}
Let $H$ be a graph, let $f\in\mathbb N$ be  even, and let $w:V(H)\to\{0,2,4,\dots,f\}$.
Then there are a stable set $Z\subseteq V(H)$ and an $f$-matching $M$ of $H$
such that
\begin{enumerate}
\item[(i)] $d_M(v)\ge w(v)$ for every $v\in V(H)\setminus Z$,
\item[(ii)]\label{lemfracii} $d_M(v)\le w(v)$ for every $v\in Z$,
\item[(iii)] $d_M(v)=f$ for every $v\in N(Z)$ and 
\item[(iv)]\label{lemfraciv} if $uv\in M$ and $u\in N(Z)$ then $v\in Z$.
\end{enumerate}
\end{lemma}

\begin{proof}
   Let  $X, Y, W$ be given by Theorem~\ref{GEthm} for $H$. 
   Let $S\subseteq Y$ be the set of all singleton components of $H-X$. Delete $W$ from $H$, and contract each component of $H[Y\setminus S]$ to a single vertex $y$ (omitting repeated edges).  Call the new graph~$H'$, and let~$Y'$ be the set of all vertices of $H'$ that arose from contractions.  By Theorem~\ref{GEthm}~(IV), the graph $H'$ has a matching $M_0$ that covers all vertices of $X$ and does not use any interior edge of $X$. By taking $\frac f2$ copies of each edge of $M_0$, we obtain an $\frac f2$-matching~$M_1$ of~$H'$ that covers~$X$. Hence, as $H'-X$ is stable, any maximum $\frac f2$-matching of $H'$  covers~$X$.
In particular, this holds for the matching 
   $M_2$, chosen among all maximum $\frac f2$-matchings of $H'$ that avoid interior edges of $X$ such that $$\sum_{v\in S}\min \{d_{M_2}(v), \frac{w(v)}2\}$$ is maximised. Let $Z'$ be the set of all $v\in S$ with $d_{M_2}(v)< \frac{w(v)}2$, and let $Z$ be the set of all $v\in S\cup Y'$ that are endpoints of an $M_2$-alternating path starting in a vertex of~$Z'$. Then by our choice of $M_2$, we know that $Z\subseteq S$, and moreover, if $v\in Z$ then $d_{M_2}(v)\le \frac{w(v)}2$. (Indeed, it is easy to see that otherwise we could use the corresponding alternating path to change $M_2$ to a `better' matching.)

    We obtain $M_2'$ from $M_2$ by replacing each edge of $M_2$ with endpoints $x\in X$, $y\in Y'$ with any   edge from $x$ to the corresponding component of $H[Y]$. We obtain $M_3$   from  $M_2'$ by taking two copies of each edge.  Then $M_3$ is an 
     $f$-matching  in $H$. 
    Furthermore, by construction,   $M_3$ and the stable set $Z$ satisfy (ii)-(iv) in $H$, and satisfy (i) for all vertices of $(S\setminus Z)\cup X$.
    
    We will now extend $M_3$ to an $f$-matching $M_4$ in $H$ by adding some edges of $H-S-N(Z)$, in order to ensure (i) for all vertices outside~$Z$. This will clearly not affect the validity of (i) for vertices of $(S\setminus Z)\cup X$, and of (ii)-(iv). To determine the edges we add to $M_3$, we consider each component~$C$ of $H[(Y\setminus S)\cup W]$ separately, and add a matching $M_C$ corresponding to that component. 
    Let $V_C$ be the (possibly empty) set  of vertices of $C$ that are touched by $M_3$. By construction, each $v\in V_C$ lies in an even number $x_v$ of edges from $M_3$, and $x:=\sum_{v\in V_C}x_v\le f$. We  obtain $M_C$  as follows: For each $v\in V_C$,   take $x_v$ copies of the perfect matching of $C-v$. The union of these $x$-matchings is $M_C'$. Note that  $d_{M_C'}(v)=x$ for each $v\in V(C)\setminus V_C$ and  $d_{M_C'}(v)=x-x_v$ for each $v\in V_C$. Also,  $x$ is even by construction. We obtain~$M_C$ by adding to $M'_C$  $(f-x)/2$ copies of a $2$-matching of $C$, which is given by  Lemma~\ref{factorcrit} if  $C\subseteq Y$, or by doubling the matching from Theorem~\ref{GEthm}~(II) if $C\subseteq W$.
    We add to $M_3$ the union of all $M_C$ where $C$ is a component of $H[(Y\setminus  S)\cup W]$. This choice~clearly~satisfies~(i)-(iv).
\end{proof}

\subsection{Embedding into  an f-matching}\label{sec:emblem}

In the proof we will use 
an $f$-matching found in the reduced graph for our embedding. Namely, we will embed  each of the small trees of an $\alpha$-decomposition of $T$ (for some small $\alpha$) into the endclusters of $f$-matching edges. In order to be able to use these endclusters evenly for all $f$-matching edges, we will always make sure to use at most $|W|/f$ vertices of each such cluster $W$ for embeddings using an edge $e$ that contains~$W$.

Further, we will often try to maintain the set of used vertices {\it balanced} for as long as possible, in a sense that will be made more precise below, but roughly means that while we embed using the edge $e$, about the same amount is embedded in each of its endpoints $W,W'$. This is advantageous for us, as in this way we will be able to fit more of the small trees into an edge. 


Let $\beta\ge 0$, and let $M$ be an $f$-matching on  $\mathcal V$, where each $W\in\mathcal V$ is a set of vertices with $|W|=m$. 
 We say  $U\subseteq \bigcup\mathcal V$ is {\it $(M,\beta)$-balanced with associated sets of nonnegative integers 
  $\{u_e^W\}_{e\in M, W\in e}$ and $\{u_W\}_{W\in\mathcal V}$} if
\begin{enumerate}
    \item $|u_e^W-u_e^{W'}|\le\beta m$ and $u^W_e, u^{W'}_e\le \frac mf$ for each $e\in M $ and $W, W'\in e$,
    \item $u_W\le (1-\frac{d_M(W)}f)m$ for each $W\in\mathcal V$, and
    \item $|W\cap U|=u_W+\sum_{e\in M, W\in e}u^W_e $ for each $W\in\mathcal V$.
\end{enumerate}
If furthermore,  $u_W=0$ for all $W\in\mathcal V$, then we say $U$ is  {\it fully $(M,\beta)$-balanced}. In this case, and if moreover, there is a set $\W\subseteq\V$ such that each $e\in M$ has exactly one endpoint in $\W$, and such that
$u_e^W\le u_e^{W'}$  for each $e\in M $ and $W, W'\in e$ with $W\in e\cap\W$,
then we say $U$ is fully $(M, \beta, \W)$-balanced. 
Note that 
for $M'\subseteq M$, an $(M, \beta)$-balanced set is also $(M', \beta)$-balanced, with associated integers $\hat u_e^W=u_e^W$ and $\hat u_W=u_W+\sum_{e\in M\setminus M', e\ni W}u_e^W$.
 
 The next lemma is meant to be used when we already have embedded parts of the tree in a balanced way.
\begin{lemma}\label{lem:T3}
Let $0<\beta\ll\varepsilon $, and let $m, f\in\mathbb N$ with $\eps\ll\frac 1{100f}$. Let $\mathcal V$ be the set of clusters of an
$(\varepsilon,5\varepsilon^{1/5})$-regular partition of
an $n$-vertex graph $G$, with $|V|=m$ for every
$V\in\mathcal V$.
  Let~$M$ be an $f$-matching of a reduced graph on $\mathcal V$. Let $Q\in\mathcal V$, and let $U\subseteq\bigcup\mathcal V$ be  $(M,\beta)$-balanced with associated sets $\{u_e^W\}_{e\in M, W\in e}$ and $\{u_W\}_{W\in\mathcal V}$. 
Let~$\mathcal T$ be a family of 
 trees on at most $\beta m$ vertices each,  such that 
 \begin{enumerate}
    \item[$(\ast)$] $ |\bigcup\mathcal T|\le \hskip-.4cm\displaystyle\sum_{\substack{W\in V(M) \\ d(Q, W)m\ge |W\cap U|}} 
    \hskip-.5cm
    \Big(
    \min\Big\{d(Q, W)m,\frac{d_M(W)}f m +u_W \Big\}-|W\cap U|\Big)
    -\epsilon^{1/15}n.$
\end{enumerate} 
 For $X\in\mathcal T$, let $Y_X$ consist of the root~$r_X$  and at most one more vertex of $V(X)$, lying at even distance at least four from  $r_X$. Let $Y=\bigcup_{X\in\mathcal T}Y_X$. Let $\psi:Y\to Q$ be  such that each $q\in\psi(Y)$ is typical to all but at most $2\sqrt\eps|\mathcal V|$  clusters of $\V$. 
 Then  there is an embedding $\varphi$ of  $\bigcup \mathcal T$   in $\bigcup V(M)\setminus U$ such that 
 \begin{enumerate}
     \item[(I)] $\varphi(y)\in N(\psi(y))$  for each  $y\in Y$  and 
     \item[(II)]
 $|\varphi(\uT)\cap W|+\sum_{e\in M, e\ni W}u^W_e\le d_{M}(W)\frac mf$ for each~$W\in V(M)$. 
 \end{enumerate} 
\end{lemma}

\begin{proof}
   In a first stage, we  embed a growing subset $\mathcal X$ of $\T$  as follows, starting with $\mathcal X=\es$. Let  $\varphi$ denote the embedding and set $U_{\mathcal X}=U\cup\varphi(\bigcup\mathcal X)$.
   If there are an  $X\in\T\setminus\mathcal X$ and an edge $e\in M$ with endpoints $W_1,W_2$ such that for each~$j=1,2$,
    \begin{enumerate}
    \item[(i)]   $v$ is typical to  $W_j$ and   $|N(v)\cap W_j\setminus U_{\mathcal X}|\ge \eps^{1/10}m$ for each $v\in \psi(Y_X)$, and 
        \item[(ii)] vertices of 
        previous trees using $e$ occupy at most $\frac mf-u_e^{W_j}-\eps^{1/10}m$ vertices of~$W_j$,
    \end{enumerate} 
   then we add $X$ to $\mathcal X$ and use Lemma~\ref{lem:T1} to embed $X$ in $(W_1\cup W_2)\setminus U_{\mathcal X}$, in a way that~(I) holds for all~$y\in Y_X$, and (II) holds for   $\varphi(\bigcup\mathcal X)$. By~$(i)$,
   we can choose to which of $W_1, W_2$ the root of $X$ goes, and thus are able to  maintain  $U_{\mathcal X}$  $(M,\beta)$-balanced throughout this stage.
   Once there is no suitable pair $X, e$, we stop embedding and set~$\T':=\mathcal X$ and $U':=U_{\mathcal X}$. We assume $\T'\neq\T$, as otherwise we are done.

Obtain $M^*$ from $M$ by discarding all edges that have an endpoint~$W_j$ failing~$(ii)$.  
 As  $U'$ is $(M,\beta)$-balanced, it is also $(M^*,\beta)$-balanced, say with associated sets $\{\hat u_W\}_{W\in\mathcal V}$ and 
  $\{\hat u_e^W\}_{e\in M^*, W\in e}$. If one endpoint $W_j$ of an edge fails $(ii)$ then also the other endpoint $W_{3-j}$  is close to failing  $(ii)$: Namely, vertices from trees 
        $X\in\T'$ that used~$e$ occupy at least $\frac mf-u_e^{W_{3-j}}-\eps^{1/10}m-\beta m$ vertices of $W_{3-j}$. Hence for each $e\in M\setminus M^*$ and each of $e$'s endpoints $W$, at least $\frac {m}f-u_e^W-\eps^{1/11}m$ vertices of~$W$ are occupied by trees of $\T'$ that use $e$. Thus, by~$(\ast)$, and setting $\T'':=\T\setminus\T'$, we have
 \begin{equation}
 \label{agosto}
 |\bigcup\mathcal T''|\le \hskip-.4cm\displaystyle\sum_{\substack{W\in V(M^*) \\ d(Q, W)m\ge |W\cap U'|}} 
    \hskip-.5cm
    \Big(
    \min\Big\{d(Q, W)m,\frac{d_{M^*}(W)}f m +\hat u_W \Big\}-|W\cap U'|\Big)-  \eps^{1/13}n.
    \end{equation}
    Now, let $X^-\in\T''$. As each $q\in\psi(Y_{X^-})$ is typical to all but at most $2\sqrt\eps|\V|$ 
  of the clusters of $V(M^*)$,  and since we did not embed $X^-$ in stage 1, there is a subset $M^-$ of $M^*$ with $|M^-|\le 4f\sqrt \eps |\V|$ such that  each $e\in M^*\setminus M^-$ has an endpoint~$W_e$ that fails~$(i)$ for $X^-$ because $d(Q,W_e)m-|W_e\cap U'|<2\eps^{1/10}m$. Let $W'_e$ be the other endpoint of $e$. Set $W^*:=\{W_e' : e\in M^*\setminus M^-\}$. 
Then,  by~\eqref{agosto}, the bound on $|M^-|$, and an easy standard
calculation, we have
    \begin{equation}
        \label{toM}
       |\bigcup\mathcal T''|\le \hskip-.4cm\displaystyle\sum_{\substack{W\in W^* \\ d(Q, W)m\ge |W\cap U'|}} 
    \hskip-.5cm
    \Big(
    \min\Big\{d(Q, W)m,\frac{d_{M^*}(W)}f m +\hat u_W \Big\}-|W\cap U'|\Big)-  \eps^{1/12}n.
    \end{equation}  
    Now we start the second stage, where we embed a growing set $\mathcal X'\subseteq\T''$, starting with $\mathcal X'=\es$. At each step,   we consider any  tree  from $\T''\setminus\mathcal X'$, which we will embed and add to $\mathcal X'$. We will also keep track of  vertices we cannot or do not wish to use, via  a set $U'_{\mathcal X'}$, which initially is equal to $U'$. 
    
    Say we are considering $X\in\T''\setminus\mathcal X'$. If there is an edge $e\in M^*\setminus M^-$ such that for each $v\in \psi(Y_X)$,    $|N(v)\cap W_e'\setminus U'_{\mathcal X'}|\ge \eps^{1/10}m$ and  $v$
       is typical to  $W_e'$, and furthermore,
    at most $\frac mf-\hat u_e^{W_e'}-\eps^{1/10}m$ vertices of~$W_e'$ are occupied by vertices used or added while embedding previous trees
using $e$ (in either stage), then we 
     embed  $X $ using Lemma~\ref{lem:T1}, 
     with the root embedded in $W_e'$. 
 We add $X$ to $\mathcal X'$,  and add $\varphi(X)$ plus an arbitrary set of $\max\{|X\cap C|, |X\cap D|\}-\min\{|X\cap C|, |X\cap D|\}$ vertices of $W\setminus U'_{\mathcal X'}$  to $U'_{\mathcal X'}$, where $W$ is the endpoint of $e$ that contains less of $\varphi(X)$. In this way, we keep the set $U'_{\mathcal X'}$ $(M^*\setminus M^-, \beta)$-balanced. Also,  (I) and (II) hold for the current set $\mathcal X'$ and the sets $Y_{X'}$ with $X'\in\mathcal X'$.
 
 Now, note that since at each step $|U'_{\mathcal X'}\setminus U'|\le 2|\bigcup\mathcal X'|$, since $U'_{\mathcal X'}$ is $(M^*\setminus M^-, \beta)$-balanced, and since the degree from $Q$ to clusters failing typicality for $\varphi(N(X))$ or not having enough free space is negligible, we can use~\eqref{toM} to see that we can always find a suitable edge $e$ for the tree~$X$ that is to be embedded in this step.  
 So we are done.
\end{proof}




\section{The proof of Theorem~\ref{maint'}.}

\label{setup}
\subsection{Choosing the constants and  a counterexample $G$}

Let $\mu<1$ be the constant from Theorem~\ref{thm:realmain}, and let
$\gamma$ be given by the premises of Theorem~\ref{maint'}. We can assume that $\gamma\le \mu^{100}/100$. 
Choose $\eps=\gamma^{10^6}$, and let $M_0, n_0$ be given by Lemma~\ref{reg:deg} for input $\eps$. 
We set $\beta:=\gamma^{300}(M_0n_0)^{-100}$ and $k_0:=\beta^{-100}$. This gives the following hierarchy of constants:
$$  \frac 1{k_0} \ll \beta \ll \frac 1{M_0n_0}\le  \eps\ll  \gamma\ll  \mu ,$$ 
where  as usual, $a\ll b$ means that there is a function which, given $b$, determines a much smaller  $a>0$. 
We will prove Theorem~\ref{maint'} for $k\ge k_0$ and  a fixed tree $T$ on $k$ vertices   by contradiction. For this, choose $G$ among all graphs with $|V(G)|\le \frac k\gamma$ and with   $d(G)>k-2$ that do not contain~$T$ such that  $|E(G)|+|V(G)|$ is minimised.  Then $G$ is robust  and~$d(G)< k$.

\subsection{Preparing $T$}\label{sec:prepT}
We use Lemma \ref{treepartlem} to obtain  a $\gamma^{50}\frac{\beta\eps}{2M_0}$-decomposition of $T$, which provides us with $S=S_C\cup S_D$,  $\mathcal T=\mathcal T_C\cup \mathcal T_D$, $\mathcal T_2$, and a labeling $C,D$ of the colour classes of $T$. In particular, we have $|S|\le \frac{100 M_0^2}{\beta^2\eps^2\gamma^{100}}=:c_\beta$; $\mathcal T_2\subseteq \mathcal T_C$; and $|\bigcup\mathcal T_C\setminus \bigcup\mathcal T_2|\ge |\bigcup\mathcal T_D|$.
Moreover,  setting $F_Z=\bigcup \mathcal T_Z$ for $Z\in \{C,D\}$, we  have 
\begin{equation}
    \label{FClarge}
    |F_C|\ge\frac{k-|S|-|\bigcup \mathcal T_2|}2+ |\bigcup \mathcal T_2|\ge \frac k2
    -c_\beta
    \text{ \ and \ } |F_D|\le \frac k2 +c_\beta,
\end{equation}
and  we set 
$$\textstyle \min_{F_Z}=\min\{|C\cap F_Z|,|D\cap F_Z|\}\text{ and } \textstyle\max_{F_Z} 
    =\max\{|C\cap F_Z|,|D\cap F_Z|\}.$$

We will now see that trees having a very special structure are very easy to embed.
Hence we can claim that $T$ does not have that structure.
\begin{claim}
    \label{unbalT}
    For each $Z\in \{C,D\}$, if 
    $\textstyle \min_{F_Z}
    \le \gamma^{32} k$
    then 
    $\textstyle\max_{F_Z} 
    \le \frac{k}{2}-\gamma^{32} k$.
\end{claim}
\begin{proof}
Assume otherwise, and let $X, Y\in \{C,D\}$ be such that $|X\cap F_Z|\ge \frac k2-\gamma^{32} k$ and $|Y\cap F_Z|\le \gamma^{32} k$. Note that
$X\cap F_Z$ contains at most $\gamma^{32} k$ parents of vertices of $Y\cap F_Z$ and at most $c_\beta$ neighbours of $S$ that are not roots. So  all other vertices of $X\cap F_Z$ are leaves of $T$, and the set $L$ of these leaves obeys $|L|\ge \frac k2-2\gamma^{32} k-c_\beta$.

Let $\tilde G\subseteq G$ with $\delta(\tilde G)>(1+\gamma^{30})\frac k2$ and $d(\tilde G)>(1-\gamma^{10})k$, as given by Lemma~\ref{mindegsubgr}, applied with $\alpha=\gamma^{10}$ and  all edge weights  1. Apply Lemma~\ref{twosides}  with $\alpha=\gamma^{10}$ to see that either $T \subseteq G$, and we are done, or there is a bipartite subgraph  $H=(A_1, A_2)$ of $\tilde G$ with $d_G(v)>(1+\gamma^{3})k$ for each $v\in A_1$ and $\delta (H)\ge \gamma^{5}k$.

We let $Z'\in\{C,D\}$ be such that $Z'\neq Z$ and  claim we can embed $$T':=T-L-(F_{Z'}\setminus\bigcup\mathcal T_2)$$ in $H$.
Indeed, to see this is possible, first note that $|T-L-F_{Z'}|\le 2c_\beta +2\gamma^{32} k<\gamma^{31} k$. Second, note that if $\bigcup\mathcal T_2\subseteq F_{Z'}$, then $Z'=C$ and $Z=D$, and thus, $|F_{Z'}\setminus\bigcup\mathcal T_2|\ge |F_Z|\ge \frac k2-\gamma^{32}k$ which means  that $|\bigcup\mathcal T_2|\le 2\gamma^{32}k$. Hence, in all cases, $|T'|\le 2\gamma^{31}k$, which means we can embed $T'$ in $H$. We can even choose which side goes to which, and we will put 
 $Y\cap V(T')$ in $A_1$ and $ X\cap V(T')$ in~$A_2$. 

We now embed $F_{Z'}\setminus \bigcup\mathcal T_2$ greedily in $\tilde G$, which is possible by the high minimum degree of $\tilde G$  and as $|T-L|\le (1+2\gamma^{32})\frac k2$. 
Finally, we embed $L$ in~$G$ which is possible because the vertices of $A_1$ have degree $(1+\gamma^{10})k$ in $G$. We embedded $T$ in~$G$, a contradiction.
\end{proof}

Together with~\eqref{FClarge}, Claim~\ref{unbalT} immediately implies:
\begin{claim}
    \label{goodminFC}
    $\min_{F_C}>\gamma^{33}k$ and $\textstyle\max_{F_D}  \le \frac{k}{2}- \gamma^{33} k$.
\end{claim}

\subsection{Preparing $G$ }\label{sec:G}

We apply Lemma \ref{reg:deg} with $\eps, n_0, M_0$ as defined above and with  $\rho=\eps^{1/100}$ to $G$ 
and let~$R$ be the corresponding  reduced graph, with parts of size $m$. By our choice of $\T$ we have $|X|\le \beta m$ for each $X\in\T$. 
Set $\kappa=k\cdot \frac{|V(R)|}{|V(G)|}$. Then $(1-\eps)k\le \kappa m\le k$ and by Fact~\ref{fact:2},
\begin{equation}\label{degreesofR}
    \text{$d(R)\ge (1-\gamma^{33} )\kappa$}.
\end{equation}
Next, we apply Lemma~\ref{mindegsubgr} to $R$  with $\alpha=\gamma^{11}$ to obtain a subgraph $R'$ of $R$ with
\begin{equation}\label{degreesofR'}
    \text{$d(R')\ge (1-\gamma^{11} )\kappa$ \ and \ $\delta(R')\ge (1+\gamma^{33} )\frac \kappa 2 $}.
\end{equation}
Let $G'$ be obtained from $G'':=G[\cup_{X \in V(R')} X]$ by deleting all of the (by Fact~\ref{fact:typical}) at most $2\sqrt\eps |G''|$ vertices of  degree below $(1+ \gamma^{34})\frac k2$. Observe that the minimum/average degree  of $G'$ is bounded by
\vskip-.2cm
\begin{equation}
    \label{degreesofG'}
    \text{
   $d(G')\ge (1-\gamma^{10})k$ and  $\delta(G')\ge (1+\gamma^{35})\frac k2$.
    }
\end{equation} 
Apply Lemma~\ref{2+alpha} to $R'$ with $\alpha=\gamma^{10}$  
to obtain  adjacent clusters $A$ and $B$ such that 
\begin{equation}\label{degA+BR'}
    \text{$d_{R'}(A)+d_{R'}(B)\ge (2+2\gamma^{10})\kappa$ \ and \  $d_{R'}(A)\ge (1 +\gamma^{10})  \kappa$},   
\end{equation} 
where without loss of generality we  assume that $d_{R'}(A)\ge d_{R'}(B)$.
To be able to use Lemma~\ref{robustlemma} in the next subsection, we will  not work in $R'$, which has slightly lower average degree than $R$. However, we are interested in the minimum degree bound of $R'$ given by~\eqref{degreesofR'}. As a compromise we consider a graph $R''$, which has almost the same average degree as $R$, and where the neighbours of $A$ and $B$ respect the minimum degree from $R'$.
Namely, we define $R''$ as the graph obtained from $R$ by deleting all edges between $\{A,B\}$ and $N(\{A,B\})\setminus V(R')$.
Observe that by~\eqref{degA+BR'},
\begin{equation}\label{degA+B}
    \text{$d_{R''}(A)+d_{R''}(B)\ge (2+2\gamma^{10})\kappa$ \ and \ $d_{R''}(A)\ge (1 +\gamma^{10} ) \kappa$}.
\end{equation}
Moreover, by construction and by~\eqref{degreesofR'},
\begin{equation}\label{degreesofR'onlyinNA}
    \text{$d(R'')\ge (1-\gamma^{10} )\kappa$, and $d_{R''}(W)\ge (1+\gamma^{33}) \frac\kappa 2$ for each $W\in N(A)\cup N(B)$}.
\end{equation}

Consider the unweighted underlying graph of $R''-A-B$. Apply Lemma~\ref{fracmatchlem} twice to this graph, both times with $f\in\{\lceil  \gamma^{-200} \rceil, \lceil  \gamma^{-200} \rceil +1\}$ even. In the first application we use weights  $w(X)= 2\lfloor \frac{ fd(A,X)}2 \rfloor $ for $X\in V(R'')$, obtaining an $f$-matching $M$ and a stable set $Z'$. In the second application we use weights  $w_{AB}(X)= 2\lfloor \frac{f\max\{d(A,X), d(B,X)\}}2 \rfloor $ for $X\in V(R'')$ and obtain an $f$-matching
 $M_{AB}$ and a stable set~$Z'_{AB}$. 
 Among all choices for $M_{AB}$, we choose $M_{AB}$ so that its intersection with $M$ is maximised.

 In the case that $d(A,  Z') \le |N(Z')|+\gamma^{33}\kappa$, the obstruction $Z'$ is not significant enough to disturb our planned use of $M$ for the embedding of~$T$, and therefore, in this case we  set $Z:=\emptyset$. In the other case we set~$Z:=Z'$. We  define $Z_{AB}$ analogously.
 Note that by~\eqref{degA+B} and by Lemma~\ref{fracmatchlem},  
\begin{enumerate}[label=(M\arabic*), ref=M\arabic*]
\item   \label{mexico}
    $ d_{R''}(A, V(M)\cup Z)\ge d_{R''}(A)-\gamma^{33}\kappa\ge \kappa+\gamma^{11}\kappa$,
    \item   \label{mexicoAB}
    $d_{R''}(A, V(M_{AB})\cup Z_{AB})+d_{R''}(B, V(M_{AB})\cup Z_{AB})\ge 2\kappa+\gamma^{11}\kappa$, 
     \item  \label{seand} if $Z\neq\es$ then $|Z|\ge
    d_{R''}(A,Z)\ge |N(Z)|+\gamma^{33}\kappa,$
     \item  \label{seandAB} if $Z_{AB}\neq\es$ and $d(A,Z_{AB})=0$ then $|Z_{AB}|\ge    d_{R''}(B,Z_{AB})\ge |N(Z_{AB})|+\gamma^{33}\kappa,$
       \item     \label{dMZ}
$d(A,W)\ge \frac{d_M(W)}f$ 
for each $W\in  Z$,
    \item     \label{McoversnbhdofA}
$\frac{d_M(W)}f\ge d(A,W)-2\gamma^{200}$ 
for each $W\in V(R'')\setminus (Z'\cup\{A,B\})$ and \\ $ \frac{d_{M_{AB}}(W)}f \ge \max\{d(A,W), d(B,W)\}-2\gamma^{200}$ 
for each $W\in V(R'')\setminus (Z'_{AB}\cup\{A,B\})$,
 %
 \item\label{M5}  each $W\in N(Z)$ is touched by $f$ edges of $M$ whose other endpoint is in $Z$, and each $W\in N(Z_{AB})$ is touched by   $f$ edges of $M_{AB}$ whose other endpoint is in~$Z_{AB}$.
\end{enumerate}

\noindent Assume $Z=\es$. We claim that  $d_{M_{AB}}(W)\ge
\min\{
d_M(W),
w_{AB}(W)\}$ for each
$W\in Z'_{AB}$.
Indeed,  suppose this does not hold for some $W\in Z_{AB}'$. Then $d_{M_{AB}}(W)<w_{AB}(W)$ and $d_{M_{AB}}(W)<
d_M(W)$, and thus there is an $e\in M\setminus M_{AB}$ with endpoints $U\in N(Z_{AB}')$,~$W$.  By Lemma~\ref{fracmatchlem}~(iii)--(iv), applied to
$M_{AB}$ and $Z'_{AB}$,  $U$ is covered by $M_{AB}$ and all its $M_{AB}$-neighbours
lie in $Z_{AB}'$. So there is an 
$e'\in M_{AB}\setminus M$ with endpoints $U$ and $U'\in Z_{AB}'$. Replacing $e'$ by $e$ increases
$|M\cap M_{AB}|$, while preserving properties~(i)--(iv) of Lemma~\ref{fracmatchlem}, since  the degree of $W$
increases by 1, but is still at most  $w_{AB}(W)$, a contradiction. Hence our claim holds and thus, 
 by~\eqref{mexico}, \eqref{McoversnbhdofA}  and~\eqref{degA+B},
\begin{enumerate}[label=(M\arabic*), ref=M\arabic*]
\setcounter{enumi}{7}
\item 
    \label{AintonewM}
   $\text{if $Z=\es$, then } 
     \sum_{W\in V(M_{AB})}
\min \big\{
d(A,W),\frac{d_{M_{AB}}(W)}f\big\}
\ge\kappa+\gamma^{12}\kappa.$
\end{enumerate}

\subsection{Finding a  powerful cluster  $B^*$ in $Z$}\label{sec:power}

The following claim will only be necessary for the embedding in Case 2.2, so it can be skipped until then if the reader prefers. 
\begin{claim}
\label{Bstar}
Let $Z^*:=Z\cap N(A)$ and $z^*:= \sum_{W\in Z^*}d(A, W)$.
Let $b>0$ be such that $d(Z^*, N(Z))\ge b|Z^*|$.
Then there are $B^*\in Z^*$ and  $\mathcal Q=\{Q_1,Q_2,\ldots, Q_{\ell}\}\subseteq N(Z)$ such that
\begin{enumerate}
    \item[(a)] $d_{R''}(B^*, \mathcal Q)\ge b$,
     \item[(b)] $d_{R''}(W, \mathcal Q)< b+1$ for all $W\in Z^*$, and 
    \item[(c)] for $1\le j<\ell$, at least one of the following holds:
    \begin{enumerate}
    \item[(i)]
    $d_{R''}(Q_j)\ge (1+\gamma^{35})\kappa  $ or
    \item[(ii)]  $d_{R''-Z^*-(N(Z)-\{Q_1,\ldots, Q_{j-1}\})}(Q_j) \ge 
      (1+\gamma^{35})  \kappa 
      -  z^*  -b $.
\end{enumerate}
\end{enumerate}
\end{claim}

\begin{proof}
    Order the clusters of $N(Z)$ as $Q_1, Q_2,\ldots, Q_{|N(Z)|}$ as follows. For $i\ge 1$, 
    if possible choose $Q_i$ such that $d_{R''}(Q_i)\ge \kappa+\gamma^{35}\kappa$, and otherwise 
    choose $Q_i$ in $N_i:=N(Z)\setminus\{Q_1,\ldots, Q_{i-1}\}$ such that $d_{R''-Z^*-N_i}(Q_i)$ is maximised.
Next, 
let $\ell\in\{1,2,\ldots,|N(Z)|\}$ be minimum such that   there is a $B^*\in Z^*$ with
$  \sum_{i=1}^{\ell}d(Q_i,B^*)\ge b.
$
Note that $\ell$ is well defined as by assumption $d(Z^*, N(Z))\ge b|Z^*|$. 
So~$(a)$ and~$(b)$~hold.

To see $(c)$, 
fix $j$ with $1\le j<\ell$ and  set $S:=\bigcup (Z^*\cup   N_j)\subseteq V(G)$. By Lemma~\ref{robustlemma}, we know that $(k-2)|S|<\sum_{v\in S}d_G(v)+ e(S, G-S)$. Thus, 
   \begin{align*}
         (1 - &\gamma^{100})  \kappa |Z^*\cup N_j|
       <
       \sum_{W\in Z^*\cup N_j}\big(d_{R''}(W)+ d_{R''-Z^*-N_j}(W)\big)
       \\
        & \le
     2d(Z^*, N(Z)\setminus N_j)+
    d(Z^*, N_j)+
     \sum_{W\in N_j}\big(d_{R''}(W) +d_{R''-Z^*-N_j}(W)\big)
        \\  & \le   
      2\min\{j,b\}|Z^*| +
     (|N(Z)|-j+1)|Z^*|+   \sum_{W\in  N_j}d_{R''}(W) +\sum_{W\in  N_j}d_{R''-Z^*-N_j}(W)  
    \end{align*}
    where the last inequality builds on
    the fact that $d(Z^*, N(Z)\setminus N_j)\le \min\{j,b\}$
    which follows from our choice of $\ell$.
     So, either there is a  $W\in N_j$ with $d_{R''}(W)\ge \kappa+\gamma^{35}\kappa$, in which case by our ordering of $N(Z)$, also $d_{R''}(Q_j)\ge \kappa+\gamma^{35}\kappa$ and  we are done, or there is a   $W\in N_j$ with 
    \begin{align*}
    d_{R''-Z^*-N_j}(W) & >\big ( \kappa 
       -\gamma^{100}\kappa -  |N(Z)|+j-1   - 2\min\{j,b\}   \big) \frac{|Z^*|}{|N_j|}-\gamma^{35}\kappa -\gamma^{100}\kappa 
      \\ & \ge 
      (1+\gamma^{35})  \kappa -z^*
        - b
    \end{align*}
      where we  used the fact that $z^*\ge |N(Z)| +\gamma^{33}\kappa$ by~\eqref{seand}, and the fact that $\gamma^{100}\kappa>1$. By our ordering of $N(Z)$, this proves the claim.
\end{proof}

\subsection{The embedding, Case 1: $Z=\emptyset$}\label{Zismepty}

Our embedding procedure splits into two main cases, depending on the behaviour of $A$, $B$, $M$, $Z$, $M_{AB}$ and $Z_{AB}$ as determined in Section~\ref{sec:G}. The current section focuses on the case that $Z=\es$, and is complemented by Section~\ref{Znotempty}, where we assume $Z\neq\es$.
In each case we show that  we can embed~$T$, which gives the desired contradiction.
 We will distinguish four subcases, which together cover all possible~cases: \smallskip \\ - Case 1.1 $d(A, V(M))\ge \kappa+ \frac{|F_D|}m + \gamma^{32} \kappa$,\\ - Case 1.2 $d(A, V(M))< \kappa+ \frac{|F_D|}m + \gamma^{32} \kappa$ and $Z_{AB}=\es$, \\ - Case~1.3 $Z_{AB}\neq\es$ and $d(A, Z_{AB})=0$,  \\ - Case 1.4 $d(A, Z_{AB})>0$. 
 \\
Depending on the case, we will rely on $M$ or $M_{AB}$,  and may use $B$~or~not. More precisely, in Case 1.1 we use $M$, and in Cases 1.2--1.4 we use $M_{AB}$.
The properties of cluster $B$ only become important in~Cases~1.2 and~1.3.

\subsubsection{Case 1.1: $d(A, V(M))\ge \kappa+ \frac{|F_D|}m + \gamma^{32} \kappa$.}  

\smallskip \
We embed $S_C$ in $A$ and $S_D$ in $B\cap V(G')$, respecting adjacencies, and with each image being typical to all but at most $2\sqrt\eps|V(R')|$ of the clusters of $R'$, which is possible by Fact~\ref{fact:typical} and since $B\in V(R')$. Let $\varphi$ denote the embedding.

Next, we  extend $\varphi$ by embedding $F_D$ greedily into $G'$  using at most $m-\gamma^{100}m\le m-f$ vertices in any cluster, which is possible by~\eqref{degreesofG'},
since each $X\in \mathcal T_D$ has~only one neighbour in~$S_D$, and since $\varphi (S_D)\subseteq V(G')$.
%
Set $U_D:=\varphi(F_D)$.
For each  $e \in M, W \in e$, choose   $\bar u_e^W \in \{\lfloor\frac {|U_D\cap W|}{f}\rfloor, \lceil\frac {|U_D\cap W|}{f}\rceil\}$ and $u_W\le (1-\frac{d_M(W)}f)m$ so that $u_W+\sum_{e \ni W } \bar u^W_e=|U _D \cap W|$. For each $e\in M$ with endvertices~$W_1,W_2$, set  
$u^{W_1}_e=u^{W_2}_e= \max \{\bar u_e^{W_1}, \bar u_e^{W_2}\}$.   Extend $U_D$ 
to a set $U$ containing exactly $u_W+\sum_{e \ni W }u^W_e$ vertices of $W$, for each $W\in V(M)$.
Then $U$ is  $(M,\beta)$-balanced and  fulfills~$|U|\le 2|F_D|$. 

To embed $F_C$ with each $y\in Y$ embedded correctly, we  apply Lemma~\ref{lem:T3} with~$Q:=A$,  the $f$-matching~$M$, the family $\mathcal T_C$ of trees,  
 sets $Y_X:=N(S_C)\cap X$, and the function $\psi$ mapping each $y\in Y$ to  $\varphi(N(y)\cap S_C)$, noting that 
by the assumption of Case~1.1,
$$ d(A,V(M))m\ge k+ |F_D|+\gamma^{51}k\ge |F_C|+2|F_D|+\gamma^{51}k\ge |F_C|+|U|+\gamma^{51}k,$$
so an easy calculation using~\eqref{McoversnbhdofA}, the
 fact that $Z=\es$ and $|V(R'')|/f=o(\gamma^{33}\kappa)$
shows that condition $(\ast)$ holds.
We embedded all of $T$. 

\subsubsection{Case 1.2: $d(A, V(M))  < \kappa + \frac{|F_D|}{m}+\gamma^{32} \kappa$  and $Z_{AB}=\es$. }

\smallskip \
As $Z=\es$ and by~\eqref{mexico}, the assumptions of this case imply that
 $d(A, V(R''))  < \kappa + \frac{|F_D|}{m}+\gamma^{31} \kappa$. Thus
by~\eqref{mexicoAB} and since $Z_{AB}=\es$, we have
\begin{equation}
    \label{dBsolarge}
    d(B, V(M_{AB}))> \frac{|F_C|}{m}+\gamma^{12}\kappa.
\end{equation}

For  $\mathcal{ T}'\in \{\T_C,\T_D\}$, if 
$|\uT'\cap D| \le  |\uT'\cap C|$, 
 set~$\rho_{\T'}:= \frac{|\uT'\cap D|}{|\uT'\cap C|}$ and $\sigma(\T'):=D$.
Otherwise, set $\rho_{\T'}:=\frac{|\uT'\cap C|}{|\uT'\cap D|}$ and $\sigma(\T'):=C$. (So $\rho_{\T'}$ is the ratio of the smaller 
to the larger colour class in~$\T'$, and $\sigma(\T')$ is the colour class that is smaller in~$\T'$.)
Further, for    $X\in\T'$ and $\T''\subseteq \T'$, set $\rho_X:= \frac{|X\cap \sigma(\T')|}{|X\setminus \sigma(\T')|}$ and $\rho_{\T''}:= \frac{|\uT''\cap \sigma(\T')|}{|\uT''\setminus \sigma(\T')|}$ (where we use the convention $a/0=\infty$ for $a>0$). 
Let $\{C^*, D^*\}=\{C,D\}$~be such that \vskip-.25cm
\begin{equation}\label{rhodos}
    \rho_{\T_{C^*}}\le \rho_{\T_{D^*}}.
\end{equation}

We embed $S_{C^*}$ in $A$ and $S_{D^*}$ in $B$, respecting adjacencies, and with each image being typical to all but at most $2\sqrt\eps|V(M_{AB})|$ of the clusters of $V(M_{AB})$. This is possible by Fact~\ref{fact:typical}.  Let $\varphi$ denote the embedding. At any point during our embedding, we call an edge $e$ {\it open} if at most $\frac mf-\gamma^{250}m$ vertices of each of its endpoints  are occupied by trees that used $e$ previously.
 
In the first step, we extend $\varphi$ to a growing subset $\mathcal X$ of $\T_{D^*}$, starting with $\mathcal X=\es$. Throughout this process, we will maintain the property that $\varphi(\bigcup\mathcal X)$ is a fully $(M_{AB}, \beta)$-balanced set. We proceed as follows: 
At each stage, as long as $\mathcal X\neq \T_{D^*}$,
we choose $X\in\T_{D^*} \setminus\mathcal X$  to minimise~$\rho_X$ if $\rho_{\T_{D^*}\setminus\mathcal X}\le 1$ and to maximise~$\rho_X$ if $\rho_{\T_{D^*}\setminus\mathcal X}> 1$. 
 If  there is  an open edge $e\in M_{AB}$ with
endpoints $W_1, W_2$ such that for each $j=1,2$,  the vertices of  $\varphi(N(X)\cap S_{D^*})$ are  typical to  $W_j$ and see at least $\gamma^{250}m$ unused vertices of~$W_j$,   then we add $X$ to $\mathcal X$ and embed $X$ by applying Lemma~\ref{lem:T1}. We take care to embed $X$ so that the embedding remains fully $(M_{AB}, \beta)$-balanced. If this is no longer possible, we end step 1  and set $\T^1_{D^*}:=\mathcal X$ and $U^1:=\varphi(\uT^1_{D^*})$. Hence after step 1, at least one of the following holds: 
\begin{enumerate}
    \item[(a)] $\T_{D^*}^1= \T_{D^*}$, or
        \item[(b)] there is an $X^-\in\T_{D^*}\setminus\T_{D^*}^1$ such that for each open   $e\in M_{AB}$ with
endpoints $W_1, W_2$ there is a $j\in\{1,2\}$ such that some $v\in\varphi(N(X^-)\cap S_{D^*})$ is not typical to  $W_j$ or sees less than $\gamma^{250}m$  vertices of~$W_j\setminus\varphi(\uT^1_{D^*})$.
\end{enumerate}
If   (a) holds, then we can   finish the embedding similarly as in Case 1.1 by  applying Lemma~\ref{lem:T3} with   $Q:=A$, $M_{AB}$,  $\mathcal T_{C^*}$,  $U:=U^1$,
the sets $Y_X$ and the function $\psi$ as above, noting that 
Condition~$(\ast)$ holds by~\eqref{McoversnbhdofA} and~\eqref{AintonewM}. 
 So from now on, we will assume that  (b) holds, and (a) does not hold. 
 
Let~$X^-$ be as in (b). As we chose the vertices of $\varphi(N(X^-)\cap S_{D^*})$ to be typical to both endpoints of all edges of $M_{AB}$, except for a set $M^-_1$ of  size $|M^-_1|\le10\sqrt\eps f|M_{AB}|$, each  open $e\in M_1:=M_{AB}\setminus M^-_1$ has at most one endpoint  in~$\mathcal V^B$, where $\mathcal V^B$ is the set consisting of all  $W\in V(M_1)$ with at most $d(B,W)m-\gamma^{50}m$ used vertices.
We let $M_2\subseteq M_1$ consist of all open edges $e$ that have an  endpoint $W_{e,B}$ in $\mathcal V^B$. 
We set $\W^B:=\V^B$ if $\sigma(\T_{D^*})=C^*$ and let $\W^B$ consist of the endpoints of $M_2$ not belonging to $\V^B$ otherwise.

By~\eqref{rhodos}, and by the way we chose $\mathcal X$, at least one of the following holds: 
$\rho_{\T_{C^*}}\le \rho_{\T_{D^*}}\le 
\rho_{\T_{D^*}\setminus \T^1_{D^*}}\le 1$
or $x:=\big||\bigcup(\T_{D^*}\setminus \T^1_{D^*})\cap C^*|-|\bigcup(\T_{D^*}\setminus \T^1_{D^*})\cap D^*|\big|\le\beta m-2$. In the former case set $\tilde\T_{D^*}:= \T_{D^*}\setminus \T^1_{D^*}$. In  the latter case let $X_=$ be a star on $x+2$ vertices, add the bipartition classes of $X_=$ to $C^*$ and $D^*$ so that 
$|\bigcup(\{X_=\}\cup\T_{D^*}\setminus \T^1_{D^*})\cap C^*|=|\bigcup(\{X_=\}\cup \T_{D^*}\setminus \T^1_{D^*})\cap D^*|$, root $X_=$ in   $X_=\cap C^*$, and set $\tilde\T_{D^*}:=\{X_=\}\cup (\T_{D^*}\setminus \T^1_{D^*})$. In both cases,
\begin{equation}
\label{tildeTD}
\tilde\T_{D^*}\neq\es,
\end{equation}
since we assume we are not in case (a). 
 In what follows we will embed $\T_{C^*}$ together with
$\tilde\T_{D^*}$ 
and if we succeed we will have embedded $T$. By our choice of $X_=$,
we have
\begin{equation}\label{ratio}
\rho_{\T_{C^*}}\le 
\rho_{\tilde\T_{D^*}}\le 1.
\end{equation}
Moreover, by~\eqref{dBsolarge} and by the definition of $\mathcal V^B$, we have
\begin{equation}\label{flor2}
\sum_{W\in \mathcal V^B}\big(d(B, W)m-|W\cap U^1|\big)\ge  |\bigcup\tilde\T_{D^*}|+\gamma^{13}k.
\end{equation}

We now start the second step of our embedding, where we extend $\varphi$ to a growing subset $\mathcal X'$ of $\T_{C^*}\cup\tilde\T_{D^*}$, starting with $\mathcal X'=\es$. 
Our plan is to use $M_2$ to embed some of 
$\tilde\T_{D^*}$
with the roots in $\mathcal V^B$, together with some of $\T_{C^*}$, 
 in a way that makes the imbalances of the embedded trees almost cancel out.
We  will always embed into unused vertices and ensure throughout  step 2 that 
\begin{equation}
    \label{bbb}
\text{$\varphi(\bigcup\mathcal X')$ is fully  $(M_2, 2\beta, \W^B)$-balanced.}
\end{equation}
Step 2 has two substeps, 2a and 2b (although we may abort after substep 2a). In sub\-step 2a, 
at each stage, we consider a pair of trees $X_1, X_2\in\tilde\T_{D^*}\setminus\mathcal X'$ with $\rho_{X_1}\le 1\le \rho_{X_2}$, if such a pair exists (we allow $X_1=X_2$). By~\eqref{flor2}, there is  an  open edge $e\in M_2$   such that for $i=1,2$,  the vertices of $\varphi(N(X_i)\cap S)$ are  typical to   $W_{e,B}$ (the endpoint of $e$ in $\V^B$)  and see at least $\gamma^{250}m$ unused vertices of $W_{e,B}$. We apply Lemma~\ref{lem:T1} to embed one of the $X_i$  with the root in $W_{e,B}$ so that  property~\eqref{bbb} is preserved, and add that $X_i$ to $\mathcal X'$ (as we choose which $X_i$ is embedded, this is always possible).  Now, if no pair $X_1, X_2$ as above exists, then either $\rho_{X}\ge 1$ for all   $X\in \tilde\T_{D^*}\setminus\mathcal X'$ and we end step 2, or 
 $\rho_{X}< 1$ for all   $X\in \tilde\T_{D^*}\setminus\mathcal X'$, and we go to substep 2b.
 
 In substep 2b, 
at each stage, we 
 consider a pair of unembedded trees $X'_1\in\T_{C^*}\setminus\mathcal X'$, $X'_2\in\tilde\T_{D^*}\setminus\mathcal X'$, if such a pair exists. Among all possible choices we choose $X_1', X_2'$ so that   $\rho_{X'_1}$ is minimised  and $\rho_{X'_2}$ is maximised.  
 If there is  an  open edge $e\in M_2$   such that   the vertices of $\varphi(N(X_2')\cap S)$ are  typical to   $W_{e,B}$  and see at least $\gamma^{250}m$ unused vertices of $W_{e,B}$, and if embedding $X_2'$ with the root in $W_{e,B}$ would preserve property~\eqref{bbb}, then we do so using Lemma~\ref{lem:T1} and add $X_2'$ to $\mathcal X'$. Otherwise, we check whether there is  an  open edge $e\in M_2$   such that   the vertices of $\varphi(N(X_1')\cap S)$ are  typical to both endpoints of $e$ and see at least $\gamma^{250}m$ unused vertices of each of them. If this is the case, and if furthermore, embedding $X'_1$ using $e$,   with $W_{e,B}$  accommodating the larger colour class of  $X'_1$  if $\sigma(\T_{D^*})=C^*$ and the smaller   colour class  if $\sigma(\T_{D^*})=D^*$,  would preserve~\eqref{bbb}, then we do this, using Lemma~\ref{lem:T1} and add $X_1'$ to $\mathcal X'$.  
 We end step 2b once a pair  $X'_1, X'_2$ as above either does not exist or  neither  of $X'_1, X'_2$ can be embedded as described.
 
 After ending step 2,
 we let $\T^2_{C^*}\subseteq \T_{C^*}, \T^2_{D^*}\subseteq \tilde\T_{D^*}$ be the set of all trees that were embedded in this step. Then by~\eqref{tildeTD} and by the way we embedded, $\T^2_{D^*}\neq\es$. Further, letting $U^2:=\varphi(\bigcup  ( \T^2_{C^*}\cup\T^2_{D^*}))$, one of the following holds
\begin{enumerate}
    \item[(I)]  $\rho_{X}\ge 1$ for all   $X\in\tilde\T_{D^*}\setminus \T^2_{D^*}$ and $\T_{C^*}^2= \es $, or
    \item[(II)]  one of $\T_{C^*}\setminus\T^2_{C^*}$,   $\tilde\T_{D^*}\setminus\T^2_{D^*}$ is empty, or 
    \item[(III)] there are $X^-_1\in\T_{C^*}\setminus\T^2_{C^*}$, $X^-_2\in\tilde\T_{D^*}\setminus\T^2_{D^*}$ and $M_2^-\subseteq M_2$ 
    with $|M_2^-|\le 10\sqrt\eps f|V(M_{AB})| $
    such that for each open $e\in M_2\setminus M_2^-$  we have one of the following:\begin{enumerate}
    \item[(A)]  $d(A,W)m< |W\cap (U^1\cup U^2)|+\gamma^{250}m$ for some endpoint $W$ of $e$; or
     \item[(B)] $d(B,W_{e,B})m< |W_{e,B}\cap (U^1\cup U^2)|+\gamma^{250}m$,
\end{enumerate}
\end{enumerate}
where for item (III), we used the fact that the vertices of $\varphi(N(X_1^-\cup X_2^-)\cap S)$ are  typical to  almost all clusters of $V(M_{AB})$.
 
 Moreover, setting  $U^2_P:=\varphi(\uT^2_P)$ for $P\in\{C^*, D^*\}$,  we claim that 
 \begin{equation}
     \label{buenosaires}
     |U^2_{C^*}|\le |U^2_{D^*}|.
 \end{equation}
 This clearly holds if $\T_{C^*}^2=\es$, so assume otherwise. We know that  $\T_{D^*}^2\neq \es$. Setting $c_1:=|U_{C^*}^2\setminus\bigcup \W^B|$, $c_2:=|U_{C^*}^2\cap\bigcup \W^B|$, $d_1:=|U_{D^*}^2\setminus\bigcup \W^B|$, $d_2:=|U_{D^*}^2\cap\bigcup \W^B|$, by the way we embedded~$\T_{C^*}^2$, we know that $c_2, d_1>0$. 
 By the way we chose $X'_1, X'_2$ and  embedded them, and by~\eqref{ratio},  we  have 
 \begin{equation*}
 \label{lluvia}
     \frac{c_1}{c_2}=\frac{|U_{C^*}^2\setminus\bigcup \W^B|}{|U_{C^*}^2\cap \bigcup\W^B|}\le \rho_{\T_{C^*}}\le \rho_{\tilde\T_{D^*}}\le \frac{|U_{D^*}^2\cap\bigcup \W^B|}{|U_{D^*}^2\setminus \bigcup\W^B|}=\frac{d_2}{d_1}\le 1,
 \end{equation*} 
 that is, $c_1d_1\le c_2d_2$. 
So by~\eqref{bbb}, which implies $c_2+d_2\le c_1+d_1$ or equivalently, $c_2-c_1\le d_1-d_2$, we have
$
c_1(d_1-d_2)\le d_2(c_2-c_1)\le d_2(d_1-d_2)
$. Thus if $d_1>d_2$, then $c_1\le d_2$, and using~\eqref{bbb} again, we see that $c_1+c_2\le d_1+d_2$, which is as desired for~\eqref{buenosaires}.
On the other hand if  $d_1\le d_2$ then  $d_1=d_2$. As we assume $\T_{C^*}^2\neq\es$, we know we did not skip substep 2b, and since all $X'_2\in\tilde\T_{D^*}$ embedded in substep 2b fulfill $\rho_{X'_2}<1$, and by~\eqref{bbb}, we cannot have embedded any of $\T_{D^*}^2$ in substep 2b. However as $d_1=d_2$ and~\eqref{bbb} that is not possible. We reached a contradiction. 
This proves~\eqref{buenosaires}.

 Note that
 $U^1\cup U^2$  is fully $(M_{AB},3\beta)$-balanced. Set  $\T_{C^*}^3:=\T_{C^*}\setminus \T_{C^*}^2$ and $\T_{D^*}^3:=\tilde\T_{D^*}\setminus  \T_{D^*}^2$. 
 
In the third and final step, we embed the remainder of $T$, according to which of the cases (I)-(III) we are in.

Let us start by assuming we are in  case (II), as this is the easiest case. We only need to embed
$\T_P^3$ for the at most one $P\in\{C^*, D^*\}$ with $\T^3_P\neq\emptyset$. For this, apply Lemma~\ref{lem:T3} with $M_{AB}$, $\T^3_P$, $Q=A$ if $P=C^*$ and $Q=B$ if $P=D^*$, $U=U^1\cup U^2$, the  neighbourhoods~$Y_X$ of~$S$ in   $X\in\T^3_P$, with $\psi$ pointing to $\varphi(N(X)\cap S)$ (and in the case of $X=X_=$ pointing to any typical vertex of $B$), and with  $
\beta_*:=4\sqrt\eps$ and 
$\eps_*:=\eps^{1/10}$ in place of its parameters $\beta$ and $\eps$.  Condition $(\ast)$ holds because we only embedded into open edges, because $U$ is  $(M_{AB},3\beta)$-balanced, and by~\eqref{AintonewM} if $Q=A$ and 
by~\eqref{flor2} and~\eqref{buenosaires} 
if~$Q=B$. 
Indeed, in the case $Q=B$ note that by~\eqref{buenosaires}, at least half  of $U^2$ belongs to $U_{D^*}^2$, while for~\eqref{flor2}, only  the degree of $B$ into $\mathcal V^B$ is counted, and $\mathcal V^B$ contains exactly one endpoint of each edge used in step 2. So while $\uT_{C^*}^2$ may block some vertices in $\bigcup N(B)$, this is compensated  by parts of clusters in $V(R'')\setminus\mathcal V^B$ that were used by~$\T_{D^*}^2$. We embedded all of~$T$.

In case (III), then we let $E_{\lnot A}$ denote the edges that are as in (III)(A). 
 We embed $\uT^3_{D^*}$ first, using Lemma~\ref{lem:T3} with $E_{\lnot A}$, $\T^3_{D^*}$, $Q=B$,  $U^1\cup U^2$, and with $Y_X$, $\psi$, $\beta_*$, $\eps_*$ as before.  Condition $(\ast)$ holds because we only embedded into open edges earlier, because  $U^1\cup U^2$  is fully $(M_{AB},3\beta)$-balanced,  
 by~\eqref{flor2} and~\eqref{buenosaires}, and by the definition of $E_{\lnot A}$.
 Let $P_1:=\varphi(\uT^3_{D^*})$. Adding vertices edge by
edge to the less occupied endpoint, we obtain a set
$P_2$ with $|P_2|\le|P_1|$ such that
$U:=U^1\cup U^2\cup P_1\cup P_2$
is fully $(M_{AB},10\beta)$-balanced.
By the definition of $E_{\lnot A}$, an easy
edge-by-edge calculation shows that passing from
$U^1\cup U^2$ to $U$ decreases the available degree from
$A$ by at most
$|P_1|+o(\gamma^{12}k)$.
So  we can embed $\uT^3_{C^*}$ as  Condition $(\ast)$ of Lemma~\ref{lem:T3} holds 
 and by~\eqref{AintonewM}, if we apply the lemma with $M_{AB}$, $\T^3_{C^*}$, $Q=A$, $U$, and with $Y_X$, $\psi$, $\eps_*$  as before and $\beta_*:=10\sqrt\eps$. 
We embedded all of~$T$.

It remains to handle case (I). We apply 
Lemma~\ref{lem:T3} with $M_2$, $\T^3_{D^*}$, $Q=B$,  $U^1\cup U^2$, and~$Y_X$ and $\psi$ as before, in the graph obtained by deleting all edges from $B$ to $\bigcup (V(M_{2})\setminus \V^B)$, to embed $\uT^3_{D^*}$.  Condition $(\ast)$ holds because we only embedded into open edges earlier, because  $U^1\cup U^2$  is $(M_{AB},3\beta)$-balanced, and by~\eqref{flor2} and~\eqref{buenosaires}. Set $U^3:=\varphi(\uT^3_{D^*})$. 
We observe that
by~\eqref{ratio} and by~\eqref{bbb}, $|C^*\cap\uT^3_{D^*}|$ and $|D^*\cap\uT^3_{D^*}|$ differ by at most $|M_2|\cdot 2\beta m\le 2\beta fn$. Since furthermore,  $\bigcup \W^B$ hosts one colour class of $\T_{D^*}^3$, and by the assumption of case (I), we can add to $U^1\cup U^2\cup U^3$ a set of at most $3\beta fn$ vertices, so that the obtained set $U$ is fully $(M_{AB},4\beta)$-balanced. 
We then embed  $\uT^3_{C^*}$ using Lemma~\ref{lem:T3} with $M_{AB}$, $\T^3_{C^*}$, $Q=A$,  $U$, and with $Y_X$, $\psi$, $\beta_*$, $\eps_*$  as before. Condition $(\ast)$ holds because we only embedded into open edges earlier, because  $U$  is $(M_{AB},4\beta)$-balanced, by~\eqref{McoversnbhdofA} and by~\eqref{AintonewM}.
Again, we embedded all of~$T$.

\subsubsection{Case~1.3: $Z_{AB}\neq\es$ and $d(A, Z_{AB})=0$.}

We use Fact~\ref{fact:typical} to embed  $S_{C}$  in $A$ and $S_{D}$ in~$B$, respecting adjacencies and  choosing images that are typical to all but at most $\sqrt\eps |V(M_{AB})|$ of the clusters of~$V(M_{AB})$.~Let~$M_Z$ consist of the edges of $M_{AB}$ that run between $Z_{AB}$ and $N(Z_{AB})$.
The hypotheses of the case together with~\eqref{seandAB}  imply that $d(B,Z_{AB})\ge |N(Z_{AB})|$.  Thus by~\eqref{degreesofR'onlyinNA},~we~have $$d(B,V(M_{Z}))m\ge |N(Z_{AB})|m\ge (1+\gamma^{33}) \frac k 2\ge |F_D|+\gamma^{34}k.$$ So condition~$(\ast)$ of Lemma~\ref{lem:T3} holds if we apply it with $M_Z$, $\T_D$, $Q=B$,  $U=\es$, and the usual neighbourhoods $N_X$ and function $\psi$ to embed $\T_D$ in edges of $M_Z$. The lemma guarantees that the set $U_B$ of all used vertices of $\bigcup V(M_Z)$ has no more than $d_{M_Z}(W)\frac mf$ vertices in each cluster $W$. So, for each edge $e\in M_Z$ with endpoints~$W,W'$, we can and will add to $U_B$ a set of $\big\lfloor|\frac{|U_B\cap W|}{d_{M_Z}(W)}-\frac{|U_B\cap W'|}{d_{M_Z}(W')} |\big\rfloor$ unused vertices  from  $W$ if $\frac{|U_B\cap W|}{d_{M_Z}(W)}<\frac{|U_B\cap W'|}{d_{M_Z}(W')}$ and from $W'$ otherwise. The new set $U'_B$  is fully $(M_{Z},\beta)$-balanced and thus also $(M_{AB},\beta)$-balanced. Apply Lemma~\ref{lem:T3} with~$Q=A$, $M_{AB}$, $\T_C$, $U=U'_B$, and the usual  $N_X$ and $\psi$ to embed~$\T_C$. Condition~$(\ast)$ holds~by~\eqref{degA+B} and since $A$ has no neighbours in $Z_{AB}$. This finishes the embedding of~$T$ and thus~Case~1.3.

\subsubsection{Case 1.4: $d(A, Z_{AB})>0$. }
For $P\in\{C,D\}$ we set $imb_P:=\big||F_{P}\cap P|-|F_{P}\setminus  P|\big|$. If $imb_C\ge imb_D$ we set $(\hat C,\hat D):=(C,D)$, otherwise we set $(\hat C,\hat D):=(D,C)$. Then 
\begin{equation}
    \label{aguadebeber}
    {\max}_{F_{\hat D}}+ {\min}_{F_{\hat C}}
   \le\frac{k-|S|}2\le
    {\max}_{F_{\hat C}}+ {\min}_{F_{\hat D}},
\end{equation}
where we use the notation $ {\max}_{F_C}, {\min}_{F_C}, {\max}_{F_D}, {\min}_{F_D}$ introduced in Section~\ref{sec:prepT}.  
Let $P^{F_{\hat C}}\in\{\hat C, \hat D\}$  be such that 
$|F_{\hat C}\cap  P^{F_{\hat C}}|\ge |F_{\hat C}\setminus P^{F_{\hat C}}|$
and let $\tilde\T_{\hat C}$ consist of all trees $X\in\T_{\hat C}$ with $ |X\cap  P^{F_{\hat C}}|> |X\setminus P^{F_{\hat C}}|$. Define $P^{F_{\hat D}}$ and $\tilde\T_{\hat D}$
analogously.

Let $M'$ consist of all edges of $M_{AB}$ between $Z_{AB}$ and $N(Z_{AB})$.
For all $e\in M'$, let $\nu_{B'}(e)$ be the endpoint of $e$ in $N(Z_{AB})$. For $e\in M'$, let $\nu_A(e)$ be the endpoint of $e$ in $N(Z_{AB})$ if  either $\max_{F_{\hat C}}=|F_{\hat C}\cap \hat C|$ and $\max_{F_{\hat D}}=|F_{\hat D}\cap \hat C|$, or $\min_{F_{\hat C}}=|F_{\hat C}\cap \hat C|$ and $\min_{F_{\hat D}}=|F_{\hat D}\cap \hat C|$, and otherwise  let $\nu_A(e)$ be the endpoint of $e$ in $Z_{AB}$. For each $X\in\T_{\hat C}$, set $\nu_X(e)=\nu_A(e)$ and for each $X\in\T_{\hat D}$, set $\nu_X(e)=\nu_{B'}(e)$. 
Further, for each $e\in M'$, we   let $W_e, W'_e$ denote its endpoints,  with $W_e\in N(Z_{AB})$ if $P^{F_{\hat D}}=\hat C$
and $W_e\in Z_{AB}$ otherwise. Let $\W$ be the set of all $W_e$, and let $\W'$ be the set of all $W'_e$.

We choose any $B'\in Z_{AB}\cap N(A)$
and use Fact~\ref{fact:typical} to embed  $S_{\hat C}$  in $A$ and $S_{\hat D}$ in~$B'$, respecting adjacencies and  choosing images that are typical to all but at most $2\sqrt\eps |V(M_{AB})|$ of the clusters of~$V(M_{AB})$. Throughout, we call the embedding $\varphi$, 
and call an edge $e\in M_{AB}$ {\it open} if  previously embedded trees that used $e$ occupy at most $\frac mf-\gamma^{250}m$ vertices of either endpoint of $e$.

We proceed in three steps. In the first step we embed  part of $\T_{\hat D}$ as follows, at any point letting~$U_1$ denote the 
set of vertices used  so far in step 1. 
At each stage of step 1,  we check whether  there is an open edge $e\in M'$ and an  $X\in\T_{\hat D}$ 
such that the vertices of $\varphi(N(X)\cap S)$ are typical to $\nu_X(e)$ and see at least $\gamma^{250}m$ unused vertices of $\nu_X(e)$. If furthermore, embedding $X$ with the root embedded in~$\nu_X(e)$ leaves 
$U_1$ fully 
 $(M',\beta,\W)$-balanced, then we do so using Lemma~\ref{lem:T1}. 
When we can no longer embed, we stop.
Note that after  step~1,
\begin{equation}\label{moll}
    \text{$\frac{|U_1|}2\ge |\varphi^{-1}(U_1)\cap P^{F_{\hat D}}|\ge \frac{|U_1|}2-2\beta mf|M'|\ge \frac{|U_1|}2-\eps n$.}
\end{equation}
Let $\T_{\hat D}^1$ be the set of  trees embedded in step 1. By~\eqref{degreesofR'onlyinNA} and by the typicality of $\varphi(S)$,~any
$X\in \T_{\hat D}\setminus\T_{\hat D}^1$ with $|X\cap\hat C|= |X\cap\hat D|$   could have been embedded, and if there
were trees $X_1, X_2\in \T_{\hat D}\setminus\T_{\hat D}^1$ with $|X_1\cap\hat C|\ge |X_1\cap\hat D|$ and $|X_2\cap\hat C|\le |X_2\cap\hat D|$ then  one of~$X_1, X_2$ could have been embedded. So either $|X\cap\hat C|> |X\cap\hat D|$ or $|X\cap\hat C|< |X\cap\hat D|$ holds  for all  $X\in\T_{\hat D}\setminus\T_{\hat D}^1$.
Hence, by the definition of $\tilde \T_{\hat D}$, and by the first inequality of~\eqref{moll},
\begin{equation}\label{mtjt}
    \T_{\hat D}\setminus\T_{\hat D}^1\subseteq\tilde \T_{\hat D}.
\end{equation}

In the second step, we 
embed some trees of $\tilde \T_{\hat C}\cup(\tilde \T_{\hat D}\setminus\T_{\hat D}^1)$, at any point letting~$U_2$ denote the 
set of vertices used  so far in step 2. 
At each stage, if
  there is an open $e\in M'$ and an  $X\in\tilde \T_{\hat C}\cup (\tilde\T_{\hat D}\setminus\T_{\hat D}^1)$ 
such that each $v\in \varphi(N(X)\cap S)$ is typical to $\nu_X(e)$ and sees at least $\gamma^{250}m$ unused vertices of $\nu_X(e)$, and if  embedding $X$  with the root embedded in~$\nu_X(e)$ maintains $U_2$ fully 
$(M', 2\sqrt\eps, \W')$-balanced, 
 then we use   Lemma~\ref{lem:T1} to embed $X$ in this way.
Once this is no longer possible, we stop, and let $\T^2_{\hat C}$ and~$\T^2_{\hat D}$ be the sets of trees embedded in step 2.
 We claim that one of the following holds:
\begin{enumerate}
    \item[(a)]  
    $\T_{\hat D}=\T^1_{\hat D}\cup\T^2_{\hat D}$, or   
    \item[(b)] for all but at most $10\sqrt\eps f|V(R'')|$ of the open edges~$e\in M'$, we have that $\displaystyle\min_{Q\in\{A,B'\}}\{d(Q,\nu_{Q}(e))m-|\nu_{Q}(e)\cap (U_1\cup U_2)|\}\le \gamma^{250}m$.
    %
\end{enumerate}

 Indeed, assume otherwise. As (a) fails, and by~\eqref{mtjt}, we have   $\es\neq \T_{\hat D}\setminus(\T^1_{\hat D}\cup \T^2_{\hat D})\subseteq\tilde \T_{\hat D}$ and hence, 
there is  an $X_1\in\tilde\T_{\hat D}\setminus(\T^1_{\hat D}\cup \T^2_{\hat D})$. By~\eqref{degreesofR'onlyinNA}, since $U_1\cup U_2$ is $(M', 3\sqrt\eps)$-balanced, and since we have not embedded all of $T$ yet,   there is a set $M''$ of at least $\sqrt\eps|M_{AB}|$ 
open edges $e\in M'$ such that the vertices of $\varphi(N(X_1)\cap S)$ are typical to $\nu_{X_1}(e)$ and see at least $\gamma^{250}m$ unused vertices of $\nu_{X_1}(e)$. However, we
did not embed $X_1$ in step~2. Thus, letting  $\{u_e^W\}_{e\in M', W\in e}$ be the set associated to the $(M',2\sqrt\eps)$-balanced set $U_2$, we have that
  for  each $e\in M''$, we have $u_e^W-u_e^{W'} \ge (2\sqrt\eps-\beta) m$, while for all other edges  $e'\in M'$ we have $u_{e'}^W-u_{e'}^{W'} \ge 0$.
 So, by the definition of~$\hat C, \hat D$, and by the second inequality of~\eqref{moll}, there is an
$X_2\in\tilde\T_{\hat C}\setminus \T^2_{\hat C}$. 
The vertices of $\varphi(N(X_1\cup X_2)\cap S)$ are atypical to  at most $10\sqrt\eps f|V(M_{AB})|$ endpoints of edges of $M'$. Since $(b)$ fails,  
there is an open edge $e\in M'$ such that the vertices of $\varphi(N(X_i)\cap S)$ are typical to $\nu_{X_i}(e)$ and see at least $\gamma^{250}m$ unused vertices of $\nu_{X_i}(e)$ for $i=1,2$. At least one of $X_1, X_2$ could   have been embedded, a contradiction. Thus at least one of~(a), (b) holds. Set $U_0:=(\varphi(S)\cap\bigcup V(M'))\cup U_1\cup U_2$.

In case $(a)$, we add some suitable unused vertices to $U_0$ to obtain a fully $(M',\beta)$-balanced (and thus also fully $(M_{AB}, \beta)$-balanced) set $U'$.
We use Lemma~\ref{lem:T3} to embed 
 $\T_{\hat C}\setminus \T^2_{\hat C}$, with $Q=A$, $U'$,  $M_{AB}$, and the usual $N_X$ and $\psi$. 
  Condition $(\ast)$ holds by~\eqref{McoversnbhdofA} and~\eqref{AintonewM}, and since by assumption $Z=\es$.

In case $(b)$, by~\eqref{degreesofR'onlyinNA}, and since $U_0$ is fully $(M', 3\sqrt\eps)$-balanced, we can  embed $\T_{\hat D}\setminus (\T^1_{\hat D}\cup \T^2_{\hat D})$  in $M'$ similarly as we embedded  $\T_{\hat C}\setminus \T^2_{\hat C}$ in case $(a)$, losing all control on the balance of its image in $M'$. 
Let $U_3$ be the vertices used for this. 
We add a set $P$ of at most $|U_3|$ vertices to $U_0\cup U_3$ so that the new set $U_4$ is fully $(M', 4\sqrt\eps)$-balanced. Thus $U_4$ is fully $(M_{AB}, 4\sqrt\eps)$-balanced.
We embed $\T_{\hat C}\setminus \T^2_{\hat C}$  by using Lemma~\ref{lem:T3}, with $Q=A$, 
$U_4$,  $M_{AB}$, the usual $N_X$ and $\psi$, and $\beta_*:=4\sqrt\eps$ and $\eps_*:=\eps^{1/10}$. Condition~$(\ast)$ holds since   $Z=\es$, by~\eqref{McoversnbhdofA} and~\eqref{AintonewM}, and as we are in case~(b), and so, using $U_3$ and adding $P$ decreases the available degree from $A$ only by $|U_3|+o(\gamma^{12}k)$. We embedded~$T$.

\subsection{The embedding, Case 2: $Z\neq\emptyset$}\label{Znotempty}

We start this case by proving that we may assume that
\begin{equation}
\label{IloveRio}
|N(Z)| <\textstyle\frac{\max_{F_C}+\max_{F_D}}m+\gamma^{34} \kappa. 
\end{equation} 
Assume the contrary (in particular $Z\neq\emptyset$). We will obtain a contradiction by embedding~$T$. Let $B'$ be a neighbour of $A$ in $Z$ (such a neighbour exists by the first part of~\eqref{seand}).  We embed   $S_C$   in $A$, and $S_D$  in $B'$, respecting adjacencies and such that each image in $A$ is typical to all but at most $\sqrt\eps |Z|$ clusters of $Z$, and each image in~$B'$ is typical to all but at most $\sqrt\eps |N(Z)|$ clusters of $N(Z)$. Fact~\ref{fact:typical} guarantees that this is possible. Let $\varphi$ denote the embedding.

By~\eqref{FClarge} and~\eqref{degreesofR'onlyinNA} we can  choose a subset $M_D$ of $\lfloor\frac fm (\max_{F_D} +\gamma^{100}k)\rfloor$  edges of $M$ between $N(B')$ and $Z$
such that for all $W \in N(Z), d_{M_D}(W) \le d(B',W)f$. 
Moreover, as we assume  (\ref{IloveRio}) does not hold, we can choose  a set $M_C$ of $\lfloor\frac fm (\max_{F_C} +\gamma^{100}k)\rfloor$ edges of $M\setminus M_D$ between $Z$ and $N(Z)$.
We plan to first embed $F_D$, for each tree  $X\in\T_D$ using an edge $M_D$ with endpoints $W \in N(Z)$, $W' \in Z$, so that $X$ is embedded in $W\cup W'$, with  its root in $W$,  
and then embed $F_C$, for each $X\in \T_C$ using an edge of $M_C$ with endpoints $W \in Z$, $W'\in N(Z)$, putting the  root of $X$ in $W$. 
\\
So, let us extend $\varphi$ to $F_D$. We successively embed trees $X\in\T_D$  by applying Lemma~\ref{lem:T1}, always maintaining the property that for each $e\in M_D$, at most $(1-\gamma^{101})\frac{m}{f}$ vertices of each endvertex of $e$ are occupied by trees using $e$ for the embedding. While doing so,  we maintain for as long as possible the property that $U_D$ is fully $(M_D, \beta)$-balanced, where $U_D$ are the currently used vertices.   By~\eqref{FClarge}, \eqref{degreesofR'onlyinNA} and our choice of $M_D$, we can maintain this property until $|X\cap C|\ge |X\cap D|$  either holds for all the remaining trees $X$ or for none of them. Hence by~\eqref{FClarge} and \eqref{degreesofR'onlyinNA}, we can embed all these remaining trees by first finding free space in the one of $\bigcup Z$, $\bigcup N(Z)$ they will use less. 
\\
To embed $F_C$ we proceed in the same way, also employing the fact that by~\eqref{dMZ}, the degree of any vertex of $\varphi (S_C)$ into any $W\in Z$ is sufficient to almost fill the $M_C$-edges incident with $W$. 
This proves~\eqref{IloveRio}.

Before we proceed, we need a quick definition. Let $C', D'\in\{C,D\}$ be such that 
\begin{equation}\label{CSDS}
|C'\setminus S|\ge |D'\setminus S|
\end{equation}
 and $C'\neq D'$. 
Clearly, $|C'|\ge \frac {k-c_\beta}2$.
We divide the remainder into three subcases, which together cover all possibilities:\smallskip \\
- Case 2.1 (a) assumes  $|F_{C'}\cap D'|\ge |F_{C'}\cap C'|$;
\\ - Case~2.1 (b) assumes  $|F_{C'}\cap D'|< |F_{C'}\cap C'|$ and $|F_{D'}\cap C'|\ge \frac k2$; \\ -  Case 2.2 assumes $|F_{C'}\cap D'|< |F_{C'}\cap C'|$ and $|F_{D'}\cap C'|< \frac k2$.\smallskip \\ Because large parts of the proofs are identical we treat Cases 2.1~(a) and~(b) together.

\subsubsection{Case 2.1: $(a)$ $|F_{C'}\cap D'| \ge |F_{C'}\cap C'| $ \\ {\color{white}{Case 2}} or $(b)$  $|F_{C'}\cap D'|< |F_{C'}\cap C'|$ and $|F_{D'}\cap C'|\ge \frac k2$. }

Since $|F_{D'}|\le k$ and $|D'|\le\frac k2 +c_\beta$, and by~\eqref{CSDS},   we know that in both subcases,
\begin{equation}\label{flor} 
\textstyle\max_{F_{C'}} < \frac{k}{2}+c_\beta\text{ \ and\ }|F_{D'}\cap C'| > |F_{D'}\cap D'|.
\end{equation}
Let $B'\in N(A)\cap Z$. We use Fact~\ref{fact:typical} to embed  $S_{D'}$  in $A$ and $S_{C'}$ in~$B'$, respecting adjacencies and  choosing images that are typical to all but at most $\sqrt\eps |V(M)|$ of the clusters of~$V(M)$. 

We will embed  $\mathcal T$ in several steps, in slightly different ways depending on the case: In Case~2.1~$(a)$
we start by embedding $F_{C'}$  from~$B'$ between $Z$ and $N(Z)$ with the roots in  $\bigcup N(Z)$, and then 
embed some of $F_{D'}$  from $A$ between $Z$ and $N(Z)$, with the roots in $\bigcup N(Z)$. In Case~2.1 $(b)$ we first embed some of $F_{D'}$  from $A$ between $Z$ and $N(Z)$, with the roots in $\bigcup N(Z)$, and then embed $F_{C'}$  from $B'$ between~$Z$ and $N(Z)$ with the roots in  $\bigcup N(Z)$, while unembedding some of $F_{D'}$ from the previous step.
Then, in both cases, we embed part of the remainder of  $F_{D'}$  from $A$ between $Z$ and~$N(Z)$,  with the roots in  $\bigcup Z$, and finally  embed the leftover of $F_{D'}$ outside~$Z\cup N(Z)$. 

Let us start step 1. 
In case $(a)$, we first
  embed $F_{C'}$ as follows. We call an edge {\it open} if each of its endpoints contains less than $\frac mf-\gamma^{250}m$ vertices that are occupied by trees that used $e$ earlier. Now, 
as long as we can find an unembedded tree  $X\in\T_{C'}$ and an open edge $e\in M$ such that the vertices of $\varphi(N(X)\cap S)$ are typical to   and see at least $\gamma^{250}m$ unused vertices of the endpoint of $e$ in $N(Z)$, and as long as embedding~$X$ keeps the  number of vertices embedded  into each endpoint of $e$ within $\beta m$, 
we  use Lemma~\ref{lem:T1} to embed $X$. Once this is no longer possible, we stop and call $\T_{C'}^1$ the set of all remaining trees from $\T_{C'}$. We claim that
 either $|X\cap D'|\ge |X\cap C'|$ for all   $X\in\T^1_{C'}$  or $|X\cap D'|< |X\cap C'|$ for all  $X\in\T^1_{C'}$. To see this, assume for contradiction there are  $X_1, X_2\in\T^1_{C'}$ with $|X_1\cap D'|\ge |X_1\cap C'|$ and $|X_2\cap D'|< |X_2\cap C'|$. Then by~\eqref{degreesofR'onlyinNA} and~\eqref{flor}, and by the typicality of $\varphi(N(X_i)\cap S_{C'})$ to most of $V(M)$,  one of~$X_1, X_2$ could  have been embedded. Thus our claim holds.
 So we are able to embed the remainder of~$\T_{C'}$ tree by tree using Lemma~\ref{lem:T1}, where for each $X\in\T^1_{C'}$, we use~\eqref{degreesofR'onlyinNA} to find an open edge $e\in M$ such that  the vertices of $\varphi(N(X)\cap S)$ 
 see at least $\gamma^{250}m$ unused vertices of the endpoint of $e$ in $N(Z)$. 
 Indeed, we can find such an edge $e$   if $|X\cap D'|\ge |X\cap C'|$ for all   $X\in\T^1_{C'}$, and otherwise, by considering the free space in the endpoints $W'\in Z$ of edges $e$ whose endpoint $W\in N(Z)$ has at least $\gamma^{250}m$ unused vertices seen by $\varphi(N(X)\cap S)$. 
 After embedding~$F_{C'}$, 
we embed as much as possible of  $\uT_{D'}$ in exactly the same way as   $\T_{C'}$,  with the roots  embedded in $\bigcup N(Z)$. We keep embedding until  we have embedded all of $\Td$ and obtained a  contradiction, or one of the~following~holds:
 \begin{enumerate}
     \item[(Ia)] $d(A, N(Z))m\le |\uTd^a\cap C'|+|\uTc\cap D'|+\gamma^{100}|N(Z)|m$ or
     \item[(IIa)]  $\frac k2\le |\uTd^a\cap C'|+|\bigcup (\Td\setminus\Td^a)\cap D'|+|\uTc\cap D'|\le\frac k2+\beta k$,
 \end{enumerate} 
  where $\Td^a$ is  the subset of $\Td$ we embedded in this step. 

Now if we are in 
  case $(b)$, we use step 1 to embed all of $F_{C'}$ and part of $F_{D'}$ as~follows. We start with $F_{D'}$, embedding as much as possible of  $\uT_{D'}$ in exactly the same way as in step 1,  with the roots  embedded in $\bigcup N(Z)$, while taking care to  never fill a cluster $W\in N(Z)$ with more than $(1-\gamma^{100})m$ vertices. This will define an auxiliary embedding~$\varphi^*$. More precisely, we embed in this way until 
  \begin{enumerate}
     \item[(Ib)] 
     $ \min\big\{\frac k2, d(A, N(Z))m\big\}\le |\uTd^*\cap C'|+\gamma^{100}|N(Z)|m$,
 \end{enumerate} holds,  where $\Td^*$ is  the subset of $\Td$ we embedded in this step. This is possible by~\eqref{flor} and as the second assumption of case $(b)$ implies that we will not run out of trees to embed unless $ |\uTd^*\cap C'|\ge\frac k2$.

Next, we begin defining the embedding $\varphi$, which starts out as $\varphi^*$ for the trees of~$\Td^*$. We embed $F_{C'}$ tree by tree as follows, while unembedding some  trees of~$\Td^*$. We keep track of the currently embedded trees of $\Tc, \Td$ by denoting them by $\Td^b$, $\Tc^b$. 
Throughout this step, we will maintain the property that for each $W\in N(Z)$, 
\begin{equation}\label{solano}
|\varphi^*(\uTd^*)\cap W|\le |\varphi (\uTd^b)\cap W|+  |\varphi (\uTc^b)\cap W|\le (1-\gamma^{101}) m.
\end{equation} 
Note that~\eqref{solano} holds before we start embedding $F_{C'}$   as then $\varphi^*(\uTd^*)=\varphi(\uTd^b)$.

  Successively for each $X\in\mathcal T_{C'}$, we will find adjacent  clusters $W\in N(Z)$, $ W'\in Z$ such that the at most two images of $N(X)\cap S_{C'}$ are typical to $W$,  such that 
  $d(B',W)m\ge |\varphi (\uTc^b)\cap W|+\gamma^{50}m$,  and such that 
      $|W'\setminus \varphi (\uTc^b\cup\uTd^b)|\ge \gamma^{250}m$. We~first find $W'\in Z\cap N(A)$, which is easy because of~\eqref{seand} and since $|\uT_{C'}\cap C'|+|\uT_{D'}\cap D'|\le\frac k2$ by the second assumption of Case 2.1~$(b)$. Then, as  on the one hand $d(W', N(Z))+ d(B',N(Z))\ge \kappa+\gamma^{33}\kappa$ by~\eqref{degreesofR'onlyinNA}, and on the other hand, 
\begin{equation*}\label{nz}
    \textstyle 
   |N(Z)|\le \frac 1m (\max_{F_C}+\max_{F_D})+\gamma^{34}\kappa \le 
\kappa-\frac 1m 
|F_{C'}\cap D'|
+\gamma^{34}\kappa
\end{equation*}  by~\eqref{IloveRio} and the first assumption of subcase $(b)$,  there is a cluster $W\in N(B')\cap N(W')$ as above, with the additional property that if $|\varphi (\uTd^b)\cap W|+  |\varphi (\uTc^b)\cap W|+|X\cap D'|> (1-\gamma^{101}) m$ then   $|\varphi (\uTd^b)\cap W|\ge 2\beta m$. 
We use Lemma~\ref{lem:T1} to embed $X$, with   $N(S_{C'})\cap X$ embedded correctly. 
 If~\eqref{solano} fails to hold after embedding~$X$, we unembed a set $\mathcal X\subseteq\Td^b$ with $\beta m\le |\bigcup \mathcal X\cap C'|\le 2\beta m$. Then~\eqref{solano} holds. 
 
 We have embedded all of $F_{C'}$. Now, we unembed some more trees of $\Td^b$, until
 \begin{enumerate}
     \item[(IIb)] $|\uTd^b\cap C'|+|\bigcup(\Td\setminus\Td^b)\cap D'|+
     |\uTc\cap D'|
     \le \frac k2 
     $
 \end{enumerate}
 holds, which is possible by~\eqref{CSDS}. 
 Note that if we unembedded at least one tree in the last step, the bound from (IIb) is close to being sharp. On the other hand, if we did not unembed any trees in the last step, then by (Ib) and~\eqref{solano}, 
$\sum_{W\in N(Z)}(|\varphi (\uTd^b)\cap W|+  |\varphi (\uTc^b)\cap W|)\ge d(A, N(Z))m-\gamma^{100}|N(Z)|m$.
So,
  \begin{enumerate}
     \item[(Ib')] either $|\uTd^b\cap C'|+|\bigcup(\Td\setminus\Td^b)\cap D'|+
     |\uTc\cap D'|
     \ge \frac k2-\beta m$ or 
     $|\varphi (\uTd^b\cap C')|+  |\varphi (\uTc^b\cap D')|\ge d(A, N(Z))m-\gamma^{100}|N(Z)|m$.
 \end{enumerate}
 This finishes step 1.
Let $\Td'$ consist of all trees of $\Td$ that were embedded in step 1. Note that in case $(a)$, $\Td'=\Td^a$, and in case $(b)$, $\Td'\subseteq\Td^b$. 

We are ready for the second step of our embedding procedure. We embed as many as possible of the  trees  in $\Td\setminus \Td'$ as follows. If there are  adjacent   $W\in Z\cap N(A)$, $W'\in  N(Z)$ and an unembedded tree  $X\in\Td\setminus \Td'$ such that both $W$, $W'$ have at least $\gamma^{250}m$ unused vertices, and if the images of $N(X)\cap S_{D'}$ have suitable unused neighbours in $W$,   then we embed $X$ into  $G[W, W']$ using Lemma~\ref{lem:T1}, with $X\cap C'$ embedded in~$W$.  Note that 
we can always find a suitable $W'$ once   $W$ is found, because of~\eqref{degreesofR'onlyinNA}, (IIa) and~(IIb).
Hence, at the end of the second step, either all of~$\Td$ is embedded, in which case we are done, or, as we will assume from now on, 
 \begin{equation}
     \label{sea}\text{after step 2, at least $d(A, Z)m-\gamma^{250}k$ vertices of $\bigcup Z$ were used.}
 \end{equation} 
 As  $d(A,Z)\ge |N(Z)|\ge \delta (R')\ge \frac\kappa 2 +\gamma^{33} \kappa$ by~\eqref{seand} and~\eqref{degreesofR'onlyinNA}, observation~\eqref{sea} implies that if $|\uTd'\cap C'|+|\bigcup(\Td\setminus\Td')\cap D'|+
     |\uTc\cap D'|\ge\frac k2-
2\gamma^{100}k$, 
     then we have embedded all of $T$. Hence,  by (Ia) and
     (Ib')
      we used at least $d(A, N(Z))m-\gamma^{100}|N(Z)|m$ vertices of $\bigcup N(Z)$.  
So by~\eqref{mexico}, we know that $d(A, V(M'))m\ge |\bigcup(\Td\setminus \Td')|+\gamma^{100}k$, where  $M'$ consists of the edges of $M$ that avoid $N(Z)$. Using Lemma~\ref{lem:T3} with   $M'$, $\Td\setminus \Td'$, $U=\emptyset$,  $Q=A$ and the usual $N_X$ and $\psi$, we  embed  $\bigcup(\Td\setminus \Td')$ in a  final step. 

\subsubsection{ Case 2.2:  $|F_{C'}\cap D'|< |F_{C'}\cap C'| $ and 
$|F_{D'}\cap C'| < \frac k2$.}

Set $Z^*:= N(A)\cap Z$ and $z^*:= \sum_{W\in Z^*}d(A, W)$. By the condition of the case $|F_{C'}\cap C'|>|C'|-\frac k2$. So since by~\eqref{seand}, $z^*\ge |N(Z)|\ge\frac \kappa2+\gamma^{33}\kappa$  we are able to choose  $\mathcal T_{C'}^{Out}\subseteq \mathcal T_{C'}$ such that 
\begin{equation}
    \label{FinN}
    |C'|- z^*m+\gamma^{41}k
    \le 
   |\bigcup\mathcal T_{C'}^{Out}\cap C'|\le  \max\{0,  |C'|- z^*m+\gamma^{40}k\},
\end{equation} 
and
moreover, by choosing those trees $X$ of $\T_{C'}$ for $\mathcal T_{C'}^{Out}$ that maximise $\frac{|X\cap C'|}{| X\cap D'|}$, and recalling
the first assumption of  Case 2.2,
we ensure that
\begin{equation}
    \label{gilberto}  |\uT_{C'}^{Out}\cap C'|\ge |\uT_{C'}^{Out}\cap D'|.
\end{equation}

Let $B^*\in Z^*$, $\mathcal Q=\{Q_1,\ldots,Q_q\}\subseteq N(Z)$ be given by Claim~\ref{Bstar} with 
$b=\frac 1m|\bigcup\TO\cap D'|+\gamma^{41}\kappa$, which fulfills the condition of the claim 
because of~\eqref{degreesofR'onlyinNA} and since $|D'|\le\frac k2+c_\beta$. 
Then  \begin{equation}\label{FinN1} 
    d_{R''}(B^*, \mathcal Q)m\ge |\bigcup\TO\cap D'|+\gamma^{42}k
\end{equation}
and for each $W\in Z^*$, \vskip-.1cm
\begin{equation}\label{FinN2}
   d_{R''}(W, \mathcal Q)m
   \le 
   |\bigcup\TO\cap D'|+\gamma^{40}k.
\end{equation} 
Moreover,
    for   $1\le j<q$, either \begin{enumerate}
        \item[(B1)] $d_{R'' }(Q_j)\ge (1+\gamma^{35})\kappa $ or 
        \item[(B2)]  $d_{R''-Z^*-(N(Z)\setminus \{Q_1,\ldots, Q_{j-1}\})}(Q_j) \ge 
      (1+\gamma^{36})  \kappa 
      -  z^*   - 
     \frac 1m |\bigcup\TO\cap D'|
     .$
    \end{enumerate}

 We will embed  $T$ in several steps, letting $\varphi$ denote  the embedding. 
Our plan is to embed  $\mathcal T_{C'}\setminus \mathcal T_{C'}^{Out}$ from $B^*$ between~$Z^*$  and $N(Z) \setminus \mathcal Q$ with the roots in $\bigcup N(Z)$, embed~$\mathcal T_{D'}$ from $A$   between  $Z^*$   and~$N(Z)\setminus\Q$  with the roots in  $\bigcup Z^*$, and finally 
 embed 
$\mathcal T_{C'}^{Out}$ from $B^*$ between $\mathcal Q$ and   $V(R'')$  with the roots in~$\bigcup \mathcal Q$. 

In the first step, we use Fact~\ref{fact:typical} to  embed $S_{D'}$ in $A$ and $S_{C'}$ in~$B^*$, respecting adjacencies and  choosing images that are typical to all but at most $\sqrt\eps |Z^*|$, resp.~$\sqrt\eps |N(Z)|$ of the clusters of $Z^*$, resp.~$N(Z)$.
Throughout we denote the embedding by $\varphi$.

In the  second step, we embed  $\mathcal T_{C'}\setminus \mathcal T_{C'}^{Out}$. Successively for each $X\in\mathcal T_{C'}\setminus \mathcal T_{C'}^{Out}$, we will find  adjacent clusters $W\in N(Z)\setminus \mathcal Q$, $ W'\in Z^{*}$ such that the  vertices of $\varphi(N(X)\cap S_{C'})$ are typical to~$W$,  such that 
  $d(B^*,W)m\ge |U\cap W|+\gamma^{50}m$,  and such that 
      $|W'\setminus U|\ge \gamma^{50}m$,
 where~$U$ are the vertices used  so far. We first find $W'$, which is easy because of~\eqref{FinN}. Then, we observe that  on the one hand $d(W', N(Z))+ d(B^*,N(Z))\ge \kappa+\gamma^{33}\kappa$ by~\eqref{degreesofR'onlyinNA}, and on the other hand,
by~\eqref{IloveRio} and by the first assumption of Case~2.2, 
$$ \textstyle  |\bigcup N(Z)|\le \max_{F_C}+\max_{F_D}+\gamma^{34}k \le 
k-
|F_{C'}\cap D'|
+\gamma^{34}k.$$ 
 So  by~\eqref{FinN2}, 
there is a suitable cluster  $W\in (N(B^*)\cap N(W'))\setminus \Q$.
We now use Lemma~\ref{lem:T1} to embed~$X$, with its   neighbour(s) of $ S_{C'}$ embedded correctly.

In the third step, we embed $\mathcal T_{D'}$ tree by tree. For each $X\in \mathcal T_{D'}$, we first use~\eqref{FinN} to find a cluster $W\in Z^*$ with at most $d(A,W)m-\gamma^{50}m$ used vertices  such that the images~of $N(X)\cap S_{D'}$ are typical to $W$.  We then find a cluster $W'\in N(W)\setminus \mathcal Q$ with at least $\gamma^{50}m$ free vertices, which is possible by~\eqref{degreesofR'onlyinNA} and~\eqref{FinN2}, since so far we only embedded  $D'\setminus(\bigcup\TO\cap D')$   in $\bigcup N(Z)$, and since $|D'|\le \frac k2+c_\beta$.   By Lemma~\ref{lem:T1}, we are able to embed $X$ correctly. 
By~\eqref{FinN}, after step 3 we have:
\begin{equation}\label{kad}
    \text{if $\TO\neq\es$ then $|\varphi(T)\cap Z^*|\ge z^*m-\gamma^{39}k$.}
\end{equation}

In the  fourth step, we embed $\TO$. 
We  embed all of $\T_{C'}^{Out}\cap D'$ in $\bigcup\mathcal Q$, while  $\T_{C'}^{Out}\cap C'$ is embedded (almost) anywhere. Our embedding proceeds in two stages. In the first stage we will identify a target cluster for $X\cap D'$ for  each tree $X\in\TO$ and embed  $N(S)\cap X$.   In the second stage we embed the rest of these trees. 
The procedure is somewhat more delicate than usual, as space is very limited, and inside $\bigcup\mathcal Q$ we have to divide the available space between both bipartition classes of~$\TO$.

Let us start the first stage. Our plan is to extend  $\varphi$ to all neighbours of~$S$ in $\bigcup\mathcal{T}_{C'}^{Out}$, and we will take the opportunity to enumerate the  trees in $\mathcal{T}_{C'}^{Out}$ as $X_1, X_2,\ldots, X_{t}$, where $t=|\TO|$. We will do this inductively for
$i=1,\ldots,t$.
For convenience, at step~$i$ call a pair $X\in\TO$, $Q_h\in\Q$  {\it $i$-good}  if  the vertices in $\varphi(N(X)\cap S_{C'})$ are typical to $Q_h$  and 
$$d(B^*,Q_h)m\ge  
     \sum_{\ell<i\land
    \varphi(X_\ell\cap N(S_{C'}))\subseteq Q_h}
    |X_\ell\cap D' |   
    +\gamma^{45}m.$$
 Setting   
 $\T_1:=\{X\in\TO:|X\cap C'|<|X\cap D'|\}$ and $\T_2:=\TO\setminus \T_1$,  we will find  an $h_i\le q$ such that   the following hold for each $i=1,\ldots,t$ (where $h_0=q-1$): 
\begin{enumerate}
    \item[(a)]  $h_i\le h_{i-1}$ 
    \item[(b)] for each $x\in X_i\cap N(S)$, we have $\varphi(x)\in Q_{h_i}$  and $\varphi(x)$ is typical to all but at most $\sqrt\eps |R''|$ of the clusters of $R''$,
    \item[(c)]  there is no $i$-good pair $X,Q_h$ with $X\in \T_1\setminus\{X_1,\ldots, X_{i}\}$ and   $h>h_i$,
     \item[(d)] there is a set $\Q'$ of at most $\sqrt\eps |N(Z)|$
     clusters of $\Q$ such that  any  $i$-good pair $X,Q_h$ with $X\in  \T_2\setminus\{X_1,\ldots, X_{i}\}$ and   $h>h_i$ fulfills $Q_h\in\Q'$,
         \item[(e)] if $X_i\in\T_1$ then for all $j<i$, $X_j\in \T_1$.
\end{enumerate}

Say we are at step $i$ of the procedure. If $\T_1\not\subseteq \{X_1,\ldots, X_{i-1}\}$, let $h_i\le h_{i-1}$ be maximum such that there is an $i$-good pair $X,Q_{h_i}$ with  $X\in \T_1  \setminus \{X_1,\ldots, X_{i-1}\}$. Such a pair exists by~\eqref{FinN1}, by~$(c)$   and by the typicality~of $\varphi(S_{C'})$ to almost all clusters of $N(Z)$. If $\T_1\subseteq \{X_1,\ldots, X_{i-1}\}$, then let $h_i\le h_{i-1}$ be maximum such that there is an $i$-good pair~$X,Q_{h_i}$ with  $X\in \T_2\setminus \{X_1,\ldots, X_{i-1}\}$. 
Such an index $h_i$ exists by~\eqref{FinN1},~$(d)$, and by the typicality~of~$\varphi(S_{C'})$.
We set $X_i:=X$. 
Our choice of $X$ and $h_i$ ensures that $(a)$ and $(e)$ hold, and we can  extend~$\varphi$ to $N(S)\cap X$ so that also  $(b)$ holds for $i$. Items~$(c)$ and $(d)$ hold for $i$ because of our choice of $Q_{h_i}$, because of~$(c)$ and $(d)$ for $i-1$, and, in the case that $\T_1=\{X_1,\ldots, X_{i}
\}$,  because at most $\sqrt\eps |N(Z)|$ clusters of $\Q$ did not form a good pair with  $X_i$. 

We proceed to the second stage where we embed the remainder of the trees  $X_1, X_2,$ $\ldots, X_{t}$ in this order. 
At each step $i$ we will~ensure~that 
\begin{enumerate}
    \item[{\small $(\clubsuit)$}] $(1-\gamma^{45})d(B^*, Q_h)m+|Q_h\cap\varphi(C'\cap \bigcup_{j\le i}X_j )|\le (1-\gamma^{50})m$ for all $h\le h_i$.
\end{enumerate}
This holds when $i=0$. Now 
say we are at step $i\ge 1$. If there is a $W\in N(Q_{h_i})\setminus \{Q_1,\ldots,Q_{h_i-1}\}$ such that each vertex of $\varphi(N(S)\cap X_i)$ sees  at least $\gamma^{50}m$ unused vertices of $W$, then we employ Lemma~\ref{lem:T1} to  embed~$X_i$ in $Q_{h_i}\cup W$ (where we use the fact that by {\small $(\clubsuit)$}, there are enough unused vertices in $Q_{h_i}$). So assume there is no such~$W$. Then by $(b)$, we know that $$d(Q_{h_i},\bigcup (V(R'')\setminus \{Q_1,\ldots, Q_{h_i-1}\}))m\le 
|\varphi(T)\setminus\bigcup_{\ell<h_i}  Q_\ell|
+\gamma^{50}n.$$
\vskip-.2cm
\noindent Thus, if
 (B1) holds for $Q_{h_i}$ then 
$$  d(Q_{h_i}, \bigcup_{\ell<h_i}  Q_\ell   )m
     \ge |X_i\cup X_{i+1}\cup\ldots X_t|+|\varphi( \bigcup_{j=1}^{i-1}X_j)\cap \bigcup_{\ell<h_i}  Q_\ell|+\gamma^{36}k,
$$
and if (B2) holds, then   by~\eqref{kad}, 
\begin{align*}
    d(Q_{h_i}, \bigcup_{\ell<h_i}  Q_\ell   )m
     &
     \ge |C'\cap (X_i\cup X_{i+1}\cup\ldots X_t)|+|\varphi( \bigcup_{j=1}^{i-1}X_j)\cap \bigcup_{\ell<h_i}  Q_\ell|+\gamma^{37}k\\
      & \ge
      |D'\cap (X_i\cup X_{i+1}\cup\ldots X_t)|+|\varphi( \bigcup_{j=1}^{i-1}X_j)\cap \bigcup_{\ell<h_i}  Q_\ell|+\gamma^{37}k,
\end{align*}  
where for the second inequality we use~\eqref{gilberto},~$(e)$, and the fact that
 $|C'\cap X|\ge |D'\cap X|$ for each  $X\in\T_2$. 
Thus there is an index $h\in\{1,\ldots,h_i-1\}$ such that  the  strengthening of {\small $(\clubsuit)$} obtained by replacing $\gamma^{50}$ by $2\gamma^{50}$ holds and 
 each of (the at most two) $v\in \varphi(N(S)\cap X_i)$ is adjacent to a set $W_1^v$ of at least $\gamma^{40}m$ unused vertices of $Q_{h}$. Let $W_1$ be the union of the sets $W_1^v$ over all $v\in \varphi(N(S)\cap X_i)$. 
Now, since~$X_i$,~$Q_{h_i}$ form an $i$-good pair, and by {\small $(\clubsuit)$},  $Q_{h_i}$ contains a set $W_2$ of at least $\gamma^{51}m$ unused vertices. By Fact~\ref{fact:1}~(ii),  the pair 
$(W_1, W_2)$ is  $\frac{2\eps}{\gamma^{51}}$-regular and has density at least $\eps^{1/99}$.  
To embed the remainder of~$X_i$, we apply Lemma~\ref{lem:T1} to the pair
$(Q_{h_i},Q_h)$, with sets $W_2$ and $W_1$
 (if $N(S)\cap X_i=\{u,u'\}$, we replace $W_2$ by
$W_2\cup\{\varphi(u')\}$ and use $u'$ as the second
prescribed vertex).
Our choice of~$h$ ensures {\small $(\clubsuit)$} continues to hold. 
This finishes the second stage and thus~Case~2.2.

\newcommand{\etalchar}[1]{$^{#1}$}


\begin{thebibliography}{99}


\bibitem{bps}
G.~Besomi, M.~Pavez-Sign\'{e}, M.~Stein.
\newblock Degree Conditions for Embedding Trees.
\newblock SIAM Journal on Discrete Mathematics, 33(3):1521-1555,
2019.


\bibitem{BPS3}
G.~Besomi, M.~Pavez-Sign\'e, M.~Stein.
\newblock On the {E}rd{\H{o}}s-{S}\'os conjecture for bounded degree trees.
\newblock {\em Combinatorics, Probability and Computing}, 30(5):741--761, 2021.


\bibitem{BurrRoberts}
S.~A. Burr and J.~A. Roberts.
\newblock On Ramsey numbers for stars.
\newblock {\em Utilitas Math.}, 4, 1973, 217--220.

\bibitem{beyond}
A. Davoodi, D. Piguet, H. {\v R}ada, N. Sanhueza-Matamala.
\newblock  The asymptotic version of the Erd{\H o}s-S{\'o}s conjecture and beyond
\newblock  Preprint arXiv 2026, arXiv:2603.17755


\bibitem{Erdos64}
P.~Erd\H{o}s.
\newblock Extremal problems in graph theory.
\newblock In {\em Theory of graphs and its applications, Proc. Sympos.
  Smolenice}, pages 29--36, 1964.


\bibitem{Erdos81}
P.~Erd\H{o}s.
\newblock Some new  problems and results  in graph theory and other branches of combinatorial mathematics.
\newblock  {\em Lecture Notes in Mathematics}, 865:9-17,1981.

\bibitem{eg75}
{ P. Erd\H{o}s and R.L. Graham.}
 \newblock On partition theorems for finite graphs.
 \newblock {\em Coll. Math. Soc. J\'anos Bolyai}, 10, 1975, 515-527.

\bibitem{HRSW20}
F.~Havet,  B.~Reed, M.~Stein, and D.~Wood.
\newblock A variant of the Erd{\"o}s-S{\'o}s Conjecture
\newblock{\em Journal of Graph Theory}, 94(1):131-158, 2020. 

\bibitem{LKS4}
J.~Hladk\'y, J.~Koml\'os, D.~Piguet, M.~Simonovits, M.~J.~Stein, and E.~Szemer\'edi,
\emph{The approximate Loebl--Koml\'os--S\'os Conjecture IV: Embedding techniques and the proof of the main result},
arXiv preprint arXiv:1408.3870 (2014).

\bibitem{LovaszPlummer}
L. Lov\'{a}sz and M.D. Plummer, \textit{Matching Theory}, Annals of Discrete Mathematics, Vol. 29, North-Holland, Amsterdam, 1986.

\bibitem{regu}
Koml{\'{o}}s, J., Shokoufandeh, A., Simonovits, M., and Szemer{\'{e}}di,
  E.
\newblock The regularity lemma and its applications in graph theory.
\newblock In {\em Theoretical Aspects of Computer Science.} 
(2000), pp.~84--112.

\bibitem{Pok24}
A.~Pokrovskiy.
\newblock {Hyperstability in the Erd{\H o}s-S{\'o}s Conjecture}.
\newblock Preprint 2024, arXiv:2409.15191.


\bibitem{RS23}
B.~Reed and M.~Stein.
\newblock Spanning trees in graphs of high degree with a universal vertex.
\newblock {\em Journal of Graph Theory}, 102(4), 797-821. 


\bibitem{RSExt}
B.~Reed and M.~Stein.
\newblock Extremal Cases of the Erd\H os-S\'os Conjecture.
\newblock {\em Preprint arXiv 2026}


\bibitem{RS25}
B.~Reed and M.~Stein.
\newblock Embedding Nearly Spanning Trees.
\newblock {\em Combinatorics, Probability and Computing},  34(6):927-931, 2025.


%

%

\bibitem{maya-survey}
M.~Stein.
\newblock Tree containment and degree conditions, in: {\it Discrete Mathematics
  and Applications}.
\newblock pages 459--486. Springer International, 2020.

\bibitem{Sze78}
E. Szemer{\'e}di.
\newblock Regular partitions of graphs.
\newblock In {\em Probl\`emes combinatoires et th\'eorie des graphes}
vol.~260 of {\em Colloq.
  Internat. CNRS}. CNRS, Paris, 1978, pp.~399--401.

\bibitem{Tut}
{ W.T. Tutte.}
 \newblock The factorization of linear graphs.
 \newblock {\em J. London Math. Soc.}, 22, 1947, 107-111.


\end{thebibliography}
\end{document}